\documentclass[dvipsnames]{article}

\usepackage{commands}
\newcommand{\rar}{\rightarrow}
\newcommand{\U}{\mathcal{U}}

\DeclareMathOperator*{\find}{find}

\newtheorem{assumption}{Assumption}[section]
\newtheorem{definition}{Definition}[section]

\usepackage{enumerate}
\usepackage{enumitem}

\newcommand{\footremember}[2]{%
    \footnote{#2}
    \newcounter{#1}
    \setcounter{#1}{\value{footnote}}%
}
\newcommand{\footrecall}[1]{%
    \footnotemark[\value{#1}]%
}

\title{Function Design for Improved Competitive Ratio in Online Resource Allocation with Procurement Costs}

\author{%
  Mitas Ray\footremember{uwece}{Department of Electrical \& Computer Engineering, University of Washington}\footnote{\texttt{mitasray@uw.edu}}%
  \and Omid Sadeghi\footrecall{uwece}%
  \and Lillian J. Ratliff\footrecall{uwece}%
  \and Maryam Fazel\footrecall{uwece}%
  }
  
\date{}

\begin{document}

\maketitle

\begin{abstract}
We study the problem of online resource allocation, where multiple customers
% decision-making agents 
arrive sequentially 
% to the platform, 
and the seller 
% platform 
must irrevocably allocate resources to each incoming customer while also facing a procurement cost for the total allocation.
% In the problem of online resource allocation with production/procurement costs, requests arrive sequentially, and an algorithm must irrevocably allocate resources to each incoming request while also facing a procurement cost for the total allocation. 
Assuming resource procurement follows an a priori 
% \maryam{maybe remove 'a priori'} \mitas{keeping to make clear that cost function is known before the arrival of any agents}
known marginally increasing 
% \maryam{is this the standard term?} \mitas{yes}
cost function, the objective is to maximize the reward obtained from fulfilling the customers' requests sans the cumulative procurement cost. 
We analyze the competitive ratio of a primal-dual algorithm in this setting, and develop an optimization framework for synthesizing a surrogate function for the procurement cost function to be used by the algorithm, in order to improve the competitive ratio of the primal-dual algorithm. 
% \maryam{how we just say in order to `improve the competitive ratio of the algorithm'? the phrase 'optimize over some class...' isn't helpful.}

% \maryam{be clear that the surrogate is for the procurement cost. also, should we convey that our class of surrogates covers what people have used in the literature (is that true, or some of the assumptions for the analysis restrict our approach?), so our designs will improve upon these?}
% \mitas{For all the literature I have seen, we are covering their approach since all they do is scale the argument. However, for production cost functions such as hinge (or smooth hinge) to be used for the budgeted problems, we do not cover this by applying the $\log$ function proposed in \cite{Eghbali2016}.}

Our first design method focuses on polynomial procurement cost functions and uses the optimal surrogate function to provide a more refined bound than the state of the art. 
% \maryam{can we say how we improve? Either say something more or remove the phrase} %result. \maryam{removed 'result'}
Our second design method uses quasiconvex optimization to find optimal design parameters for a general class of procurement cost functions. 
Numerical examples are used to illustrate the design techniques. 
We conclude by extending the analysis to devise a posted pricing mechanism in which the algorithm does not require the customers' preferences to be revealed.
%  an additional primal-dual algorithm which computes the primal and dual variables sequentially, yielding a natural economic interpretation of posted pricing.
%  We conclude by proposing an additional primal-dual algorithm which computes the primal and dual variables sequentially, which gives a much lower computational complexity, but worse competitive ratio guarantees than the algorithm previously discussed.
%  \lily{This sentence is not really informative. you didnt really say anything useful about the "additional primal-dual algorithms" which makes me know what you are talking about or want to read further.}

% \maryam{Just some random musing, but maybe we can make clear that: we view a variety of existing algorithms as essentially equivalent to applying primal-dual updates, but with different surrogates for the procurement cost. In other words, instead of thinking ``let's design a new algorithm with better competitive ratio", our approach shifts the question to designing new surrogates. This general idea is introduced for online algorithms in [Eghbali16], and is close to how choosing different regularization functions in online learning, or Bregman divergences in online mirror descent, based on the problem geometry, can impact the performance of algorithms.}

% \maryam{Also, we should be clear that we use the surrogate in the algorithm, but compare the algorithm performance to the best offline with the ``original" cost (this should be obvious, but sometimes people think introducing a surrogate means we just change the problem and our goal altogether.)}

\end{abstract}

% \maryam{I think you were planning to change the 'auction' terminology; doesn't fit.}

\section{Introduction}
\label{sec:introduction}
In the problem of online resource allocation, a seller is allocating $D$ types of resources to $T$ incoming customers. The $t$-th customer has a payment function, denoted $v_t:\R_+^D\rar\R_+$, which satisfies a natural set of assumptions listed in Assumption \ref{assump:v}. The payment function reveals how much a customer will pay for any assigned bundle of resources. The seller has a procurement cost function, denoted by $f:\R^D_+\rar\R_+$, which represents the cost incurred by the seller in procuring the resources in the cumulative allocation, and is known to the seller a priori. The procurement cost function satisfies Assumption \ref{assump:f}. We use $\vec{x}_t\in [0,1]^D$ to denote the bundle allocated to customer $t$, where the $d$-th entry of $\vec{x}_t$ represents the amount of the $d$-th type of resource in this bundle. The goal of the seller is to maximize the revenue collected from the assigned bundles to the customers minus the procurement cost of the cumulative allocation. Had the seller known the $T$ customers' payment functions beforehand, then the optimal allocation would be the result of the following \textit{offline} optimization problem:

% The offline optimization problem is the following:
\begin{equation}
\label{eq:primal} 
\begin{array}{ll}
\maximize& \sum_{t=1}^T v_t\fp{\vec{x}_t} - f\fp{\sum_{t=1}^T \vec{x}_t}\\
\mbox{subject to}& \vec{0} \preceq \vec{x}_t \preceq \vec{1} \spforall t\in[T]\\
\end{array},\tag{P-1}
\end{equation}
where $\vec{x}_t \in [0,1]^D$ for $t=1,\ldots, T$ is the optimization variable. The challenge of online resource allocation comes from its \textit{online} nature; the seller does not have any knowledge of the future customers and the seller must make an irrevocable allocation upon the arrival of each customer.
% and $v_t:\R^D_+\rar\R_+$ satisfies the following assumptions. 

% In this paper, we study an online resource allocation problem. Here, a seller has a set of $D$ types of resources and chooses to allocate the items to an incoming set of $T$ customers who arrive one by one, where each customer is allocated a bundle of resources. 
% Each customer is allocated a bundle which may contain no more than a single unit of a particular type of resource. 
% These resources can be obtained by the seller at an increasing marginal cost which is captured by a procurement cost function known offline. 
% Each customer has a payment function denoting their payment for any assigned bundle. The seller must make irrevocable decisions on how to allocate resources to each customer without knowledge of future customers. The objective of the seller is to maximize the revenue accrued from assigning these bundles minus the procurement cost of the cumulative allocation. 
% The generality of this framework is discussed in \ref{sec:examples}.

Online resource allocation has been studied extensively in the setting of fixed resource capacities \cite{Blumrosen2007,Balcan2008,Chakraborty2013},
% \maryam{how about OR literature? can we have a short paragraph reviewing OR lit that do online resource allocation, eg, works by Patrick Jaillet?} 
where there is a hard budget on each type of resource, and the unlimited supply setting \cite{Balcan2005,Balcan2008}, where the seller has unlimited access to each resource type. 
% \maryam{Blum paper is listed for both of these, does it study both?}\mitas{yes}
However, in many real-world situations, additional resources 
%(e.g., additional skilled labor or materials) 
may be procured albeit at increasing marginal costs, such as energy costs for running computer processors \cite{Makarychev2014,Andrews2016} and hiring costs for skilled labor \cite{Blatter2012}. This motivates the problem of online resource allocation with procurement costs \cite{Blum2011}. 

Past literature \cite{Blum2011, Huang2018} considers the procurement cost function to be \textit{separable}---i.e., the total cost incurred is the sum of the individual procurement costs for each resource. 
% In considering the example of a data market, if multiple experts are working on a similar task, then there may be multiple copies of resources needed. In this regard, procuring a particular resource may influence the cost of procuring another, and so the cost from assignments of experts are coupled.
The work of \cite{Blum2011}, which was further tightened in \cite{Huang2018}, 
% study this problem in the framework of an online auction. In this setting, there is a seller with a set of items that they are willing to allocate to customers who arrive one by one with value functions over bundles of items. The goal of the seller is to maximize the sum of the customers' value for the bundles they get minus the total cost incurred in producing all bundles. 
propose an online mechanism
% , called posted pricing, 
in which the seller determines the price of a particular item as a function of how much has already been sold, and the customer then chooses the bundle that maximizes their valuation function. In both works, it is assumed that the procurement cost function is separable, and so the cost of producing one item, has no effect on the cost of producing another. However, in a real-world setting there may exist limited procurement infrastructure where procuring one resource would increase the cost of procuring another. It is important then, to generalize this setting and consider procurement cost functions that are \textit{non-separable}.

While the work of \cite{Chan2015} studies the setting of non-separable procurement costs, the assumptions that the authors make essentially restricts the procurement cost function to polynomials. Therefore, the class of separable procurement cost functions has not been truly generalized, and there is no strategy on how to handle procurement cost functions which do not meet their stringent assumptions. We, on the other hand, in Theorem \ref{thm:quasiconvex}, drop the assumptions that restrict the function class to polynomials allowing us to consider general non-separable procurement cost functions. 

% by solving an optimization problem which is quasiconvex in the parameters that construct a surrogate function we use in our algorithm in order to optimize the competitive ratio. \mitas{say this same thing without mentioning surrogate functions or algorithms, unless previously stated}
% and optimize the performance of our algorithm 
% \maryam{algorithm itself is straighforward primal-dual; so maybe we can remove 'proposed' here?} 
% by directly solving an optimization problem which is quasiconvex in the parameters that make up the surrogate function.
% \maryam{make this more concrete. point to the actual theorem in the paper, and give a high-level description.}

% In the setting with non-separable procurement costs, Chan et al. \cite{Chan2015} propose the performance of a primal-dual algorithm for the online resource allocation problem. However, the assumptions that the authors make on the procurement cost function essentially restricts the function class to polynomials. On the other hand, in Section \ref{sec:surrogate}, we consider general procurement cost functions and optimize the performance of our proposed algorithm by directly solving a quasiconvex optimization problem.

Many algorithms in this setting are primal-dual algorithms, which comes from updating the dual variable at each time step and using it to assign the primal variable \cite{Buchbinder2007, Buchbinder2009, Devanur2012, Agrawal2014, Azar2016, Eghbali2016}. 
% The key step is formulating the update strategy for the dual variable and its relation to the primal allocation. 
One measure of algorithm performance in online optimization is the competitive ratio, which is defined as the ratio of the objective value achieved by the algorithm to the offline optimum (see Section \ref{sec:competitive_ratio}). The competitive ratio we consider is under the adversarial arrival order, where the seller has no knowledge of the arriving customers or the order of their arrival. 
%, besides mild assumptions on their individual payment functions
For more details on different arrival models, we refer readers to Section 2.2 in \cite{Mehta2013}.

The problem of online resource allocation appears often in the operations research community for problems like airline revenue management \cite{hwang2018online, jaillet2012near}, hospital appointment scheduling \cite{legrain2013stochastic, erdogan2015online}, and bidding in auctions \cite{bertsimas2009optimal}, amongst others. Many of the underlying assumptions in these problems, are different however, from the ones we make in our setting. For example, \cite{hwang2018online} considers the arrival time of a fraction of agents to be chosen by an adversary, while the remaining agents come at random times. The optimization problems are also formulated differently for each setting; for example, \cite{legrain2013stochastic} considers a linear objective with budget constraints. Nonetheless, these setups encourage us to scrutinize our assumptions in order to capture many problem settings. Section \ref{sec:examples} enumerates a few motivating applications of the framework proposed in this paper. For more details on related work, see Section \ref{sec:related}.

\subsection{Contributions}
We analyze a greedy primal-dual algorithm, formalized in Algorithm \ref{alg:sim_update} in which a surrogate function is used in place of the known procurement cost function in order to optimize the performance of the algorithm. We discuss a simple example in Section \ref{sec:numerical} to show that the competitive ratio of the greedy primal-dual algorithm without a surrogate function approaches zero asymptotically, illustrating the necessity of a surrogate function. Our main contributions come in the design of the surrogate function.
\begin{itemize}
    \item For polynomial procurement cost functions, we design a surrogate function to be used in the algorithm that achieves a better competitive ratio than the state of the art result in \cite{Chan2015}. Furthermore, due to a lower bound result in \cite{Huang2018} (Theorem 10), we know that our result is tight. Our result is stated formally in Theorem \ref{thm:polynomial}.
    \item For general procurement cost functions, we write the surrogate function design problem as a quasiconvex optimization problem in which the optimization variables define the function. This strategy comes from adopting an optimization perspective for maximizing the competitive ratio similar to \cite{Eghbali2016}. This technique allows us to construct surrogate functions for a wide class of procurement cost functions beyond those that are separable \cite{Huang2018} and polynomial \cite{Chan2015}. Our result is stated formally in Theorem \ref{thm:quasiconvex}.
    \item We propose Algorithm \ref{alg:seq_update_offset} which computes the primal and dual variables one after the other, as opposed to solving a saddle-point problem, as does Algorithm \ref{alg:sim_update}. 
    % Algorithm \ref{alg:seq_update_offset} is then a posted pricing mechanism, and therefore, satisfies incentive compatibility. 
    We extend the quasiconvex surrogate function design technique to this algorithm. Our results are stated formally in Theorems \ref{thm:quasiconvex_seq_0} and \ref{thm:quasiconvex_seq_1}. 
\end{itemize}
% \maryam{arrange these as a bulleted list}
% We analyze a greedy primal-dual algorithm in which a surrogate function is used in place of the known procurement cost function.
% that is tuned for the known procurement cost function. \maryam{this may be confusing, maybe remove it.}
% We propose design techniques for a surrogate function to be used by this algorithm to directly optimize the competitive ratio bound. In contrast to previous work, we view this problem from an optimization perspective and write the surrogate function design problem as a quasiconvex optimization problem. \maryam{Reza's older paper has used the same idea, you can cite it and say 'using a similar approach to the one in... here we...'} This allows us to construct surrogate functions for a wide class of procurement cost functions beyond those that are separable \cite{Huang2018} and polynomial \cite{Chan2015}. 

% We also devise a posted pricing mechanism based on the aforementioned greedy primal-dual algorithm, 
% \maryam{what do you mean by we extend the `analysis' to devise a mechanism?}
% in which at every time step, the algorithm only determines the price of each resource, and the customer decides their allocation upon arrival. \maryam{so what is the contribution of this part?}

We complement our theoretical results with simulations in which we implement our design techniques on a numerical example and better performance over the current state of the art.

% We then compare the performance of the algorithm for our design techniques on an additional example. \maryam{need something more concrete to say here.}

\subsection{Organization}
This paper is organized as follows. We close this section with a few motivating examples to show the generality of our framework. Then, we discuss preliminaries on convex optimization in Section \ref{sec:preliminaries}. We introduce the formal problem statement and primal-dual algorithm in Section \ref{sec:problem}. We analyze the competitive ratio for our primal-dual algorithm in Section \ref{sec:analysis}. We then propose our surrogate function design techniques in Section \ref{sec:surrogate}. In Section \ref{sec:numerical}, we implement our design techniques on a numerical example. In Section \ref{sec:posted_pricing}, we extend the competitive analysis and design techniques to another primal-dual algorithm which computes the primal and dual variables sequentially.
% , which possesses a natural economic interpretation of posted pricing

\subsection{Motivating Examples}
\label{sec:examples}
To illustrate applicability, we provide a number of online resource allocation problems which can be cast in the proposed framework described in Problem \eqref{eq:primal}. In each application, we describe the incoming valuation functions, $\vec{v}_t$, and the cost function, $f$. We also describe what our decision vector at time $t$---i.e., $\vec{x}_t$, represents.

\paragraph{Online auction.} A seller has a set of $D$ items and $T$ customers arrive sequentially. Let $\vec{x}_t\in [0,1]^D$ represent the decision vector at time $t$ representing the bundle allocated to customer $t$. Each item can be included in a bundle a maximum of once. Hence, the decision vector is constrained to $\vec{0}\preceq \vec{x}_t \preceq \vec{1}$. The payment function $v_t:\R_+^D\rar\R_+$ is revealed by the $t$--th customer upon arrival.
The procurement function is denoted $f:\R_+^D\rar\R_+$. The objective of the seller is to maximize their profit---i.e., the sum of the payments of the customers minus the procurement cost of the total allocation. Variations of this framework are discussed in \cite{Bartal2003, Chan2015, Huang2018}.

\paragraph{Data market.} A manager supervises a set of $D$ experts with differing expertise. Data analysis tasks, such as classifying medical images, arrive online sequentially and each task can be assigned to any subset of the experts. Upon arrival, task $t$ reveals a vector $\vec{c}_t$ where $\bracket{\vec{c}_t}_d$ quantifies the value that expert $d$ would provide the manager if assigned to task $t$. The value function is linear---i.e., $v_t\fp{\vec{x}_t}=\vec{c}_t^\top\vec{x}_t$. When a task is assigned to an expert, the amount of time they are being paid to spend on it is bounded. Therefore, the decision vector is constrained to $\vec{0}\preceq \vec{x}_t \preceq \vec{1}$. The manager is responsible for paying for the experts' time and the resources needed for the experts to do their work, which is captured in a cost function $f:\R_+^D\rar\R_+$. The cost of hiring skilled labor is marginally increasing and follows a convex cost function \cite{Blatter2012}. The objective of the manager is to maximize the value of the completed work minus the cost of getting the work completed. Variations of this application are mentioned in \cite{Ho2012}.

\paragraph{Network routing with congestion.} A network routing agent has a set of $D$ pairs of terminals and $T$ users arrive online with valuation functions over these routed connections. Since each pair of terminals can be assigned to a user at most once, the decision vector $\vec{x}_t$ is constrained to $\vec{0}\preceq \vec{x}_t \preceq \vec{1}$. Let $v_t:\R_+^D\rar\R_+$ represent the payment function that each customer reveals upon arrival and let $f:\R_+^D\rar\R_+$ denote the congestion cost function which can represent the energy needed to maintain the routed connections. Since energy usage follows dis-economies of scale---i.e., energy usage is super-linear in terms of processor speed \cite{Makarychev2014,Andrews2016}, $f$ satisfies Assumption \ref{assump:f}.
The objective of the network routing agent is to maximize the valuations of the customers minus the energy costs of the cumulative assignment. Variations of this framework are discussed in \cite{Blum2011}.

\section{Preliminaries}
\label{sec:preliminaries}
In this section, we review mathematical preliminaries as needed for the technical results.  
%Further details on common definitions in convex analysis can be found in \cite{10.5555/993483}. \lily{not. sure this is needed unless you are referencing somethign specific. then put the citation where that thign comes up and point where in the ciation you are pointing the reader to}

%\paragraph{Notation.} 
% 
Throughout, we will use boldface symbols to denote vectors. For a $D$-dimensional vector $\vec{u}\in \R^D$, let $u_i$, or equivalently $\bracket{\vec{u}}_i$,  denote the $i$-th entry. The inner product of two vectors $\vec{u},\vec{v}\in \R^D$ is denoted  $\inprod{\vec{u}}{\vec{v}}$ or,  equivalently, $\vec{u}^\top\vec{v}$.
The generalized inequality with respect to the non-negative orthant is denoted $\vec{u}\succeq\vec{v}$, and is equivalent to $u_i\geq v_i$ for all $i$.
Define $\1{\vec{u}\succeq\vec{v}}$ to be 
the  vector where the $i$-th entry equals one if $u_i\geq v_i$ and zero otherwise.

Several function properties are need for the analysis in this paper. A function $f:\R^D\rar\R$ is separable if it can be written as
\[f\fp{\vec{u}}= \sum_{i=1}^Df_i\fp{u_i}.\]
A function $f:\R^D\rar\R$ is convex if $\dom f$ is a convex set and for all $\vec{u},\vec{v} \in \dom f$, \[f\fp{\theta \vec{u} + \paren{1-\theta}\vec{v}} \leq \theta f\fp{\vec{u}} + \paren{1-\theta} f\fp{\vec{v}}\]
for any $\theta \in \bracket{0,1}$.
A function $f:\R^D\rar\R$ is quasiconvex if $\dom{f}$ is a convex set and for each $\alpha\in\R$, the sub-level set, $S_\alpha=\setcond{\vec{u}\in\dom{f}}{f\fp{\vec{u}}\leq\alpha}$ is a convex set. A function $f:\R^D\rar\R$ is closed if for each $\alpha\in\R$, the sub-level set, $S_\alpha$, is a closed set. 

Given a function $f:\R^D\rar\R$, its convex conjugate $f^\ast:\R^D\to \R$ is defined be \[f^*\fp{\vec{v}} = \sup_{\vec{u}} \vec{v}^\top \vec{u} - f\fp{\vec{u}}.\] For any function $f$ and its convex conjugate $f^\ast$, the Fenchel-Young inequality holds for every $\vec{u},\vec{v}\in \R^D$: \begin{equation}
    f\fp{\vec{u}}+f^*\fp{\vec{v}}\geq \vec{u}^\top\vec{v}.\label{eq:fenchel}
\end{equation} For a differentiable, closed and convex function $f$, its gradient is given by $\nabla f\fp{\vec{u}} = \argmin_{\vec{v}} \vec{v}^\top\vec{u} - f^*\fp{\vec{v}}$, and furthermore,  $f^{**} = f$. Letting $\vec{v}=\nabla f\fp{\vec{u}}$, the Fenchel-Young inequality holds with equality:
\begin{equation}
f\fp{\vec{u}}+f^*\fp{\nabla f\fp{\vec{u}}} = \vec{u}^\top\nabla f\fp{\vec{u}}.
\label{eq:fenchel-young}
\end{equation}

Similarly, given a function $g:\R^D\rar\R$,  the concave conjugate $g_*:\R^D\rar\R$  is defined by \[g_*\fp{\vec{v}} = \inf_{\vec{u}} \vec{v}^\top \vec{u} - g\fp{\vec{u}}.\]
An analogous inequality to \eqref{eq:fenchel} holds:
$g\fp{\vec{u}}+g_*\fp{\vec{v}}\leq \vec{u}^\top\vec{v}$ for all $\vec{u},\vec{v}\in \R^D$.
 For a differentiable, closed and concave function $g$, its gradient is given by  $\nabla g\fp{\vec{u}} = \argmax_{\vec{v}} \vec{v}^\top\vec{u} - g_*\fp{\vec{v}}$ and, furthermore,   $g_{**} = g$. Again, with $\vec{v}=\nabla g\fp{\vec{u}}$,  Fenchel-Young inequality with equality:
\begin{equation}
g\fp{\vec{u}}+g_*\fp{\nabla g\fp{\vec{u}}} = \vec{u}^\top\nabla g\fp{\vec{u}}.
\label{eq:fenchel-young-concave}
\end{equation}

Finally, the index set $\set{1,\dots,K}$ is denoted $\bracket{K}$.

\section{Problem Statement}
\label{sec:problem}
% \lily{More context here. What is the set up? say somethign abotu what you are going to formulate, online and offline optimization problems. what the variables mean? are they resources? what does $x$ represent and f and v and so on ....}\lily{Related to comment** below, write what the goal is in terms of designing an algorithm that can solve the online problem. describe the challenges: i.e. what is unknown to the algorithm. describe that often the orignal cost function can lead to poor performance relative to the offline problem for an algorithm and hence, we want to deisgn a surrogate. Connect the things like regularization in other fields and its roll in improving performance of the learning or optimization process. }
To begin this section, we formalize the problem statement described in Section \ref{sec:introduction} by explicitly describing the online and offline components as well as the assumptions made on the payment functions of the customers and the procurement cost function of the seller. 

As described in Section \ref{sec:introduction}, had the seller known all the customers that were to arrive, they would solve Problem \eqref{eq:primal} to obtain the optimal allocation to make to each customer. We denote the optimal value of Problem \eqref{eq:primal} as $P^{\star}$. However, the challenge faced by the seller is that they have no knowledge of future customers, and so the seller must make decisions that trade off making a hefty profit now with saving resources to potentially make a larger profit later. The seller knows the procurement cost function, $f$, prior to any customers arriving. Upon arrival, the customer reveals their payment function, $v_t$, and the seller must then make an irrevocable allocation before interacting with the next customer. In Section \ref{sec:posted_pricing}, we discuss an algorithm which does not require the customer to reveal their payment function. We have the following assumptions on the payment function of each customer. 

\begin{assumption}
\label{assump:v}
The function $v_t:\R_+^D\rar\R_+$ satisfies the following:
\begin{enumerate}[itemsep=0pt]
\item $v_t$ is concave, differentiable, and closed.
\item $v_t$ is increasing; i.e., $\vec{u}\succeq\vec{v}$ implies that $v_t\fp{\vec{u}}\geq v_t\fp{\vec{v}}$.
\item $v_t\fp{\vec{0}}=0$. % need this assumption for when we apply first order condition of concavity of $v_t$ in bounds by setting $x_t=0$
\end{enumerate}
\end{assumption}
Concavity in Assumption \ref{assump:v}(1) comes from the idea that a customer is willing to pay marginally less for a larger bundle, which comes from the natural desire for the customer to receive a \textit{bulk discount}. Assumption \ref{assump:v}(2) reflects the customer's willingness to pay more for a larger bundle and Assumption \ref{assump:v}(3) states that a customer would pay nothing for an empty bundle.

The procurement cost function satisfies the following assumptions.
\begin{assumption}
\label{assump:f}
The function $f:\mathbb{R}_+^D\to \mathbb{R}_+$ satisfies the following:
\begin{enumerate}[itemsep=0pt]
\item $f$ is convex, differentiable, and closed.
\item $f$ is increasing; i.e., $\vec{u}\succeq\vec{v}$ implies that $f\fp{\vec{u}}\geq f\fp{\vec{v}}$.
\item $f$ has an increasing gradient; i.e., $\vec{u}\succeq\vec{v}$ implies that $\nabla f\fp{\vec{u}}\succeq \nabla f\fp{\vec{v}}$.
\item $f\fp{\vec{0}}=0$. % is normalized; i.e., it is zero at the origin.
\end{enumerate}
\end{assumption}
Convexity in Assumption \ref{assump:f}(1) along with Assumption \ref{assump:f}(3) captures the idea that procuring scarce resources comes at an increasing marginal cost. Assumption \ref{assump:f}(2) comes from a larger cumulative allocation incurring a larger production cost and Assumption \ref{assump:f}(4) states that the seller incurs no cost for allocating nothing.
% Assumption \ref{assump:f}(3) is equivalent to super-modularity for continuous functions \cite{Krause2014}.
% \lily{why is this important to point out? seems just like a random observation? eitehr comment more as to why this is relevant or how it connects our work to others, or drop it}

\subsection{Performance Metric}
\label{sec:competitive_ratio}
% \lily{**Either make clear here or in teh section above the examples, what the algorithm is. you describe a decision maker in the examples but not an algorithm.} 
The performance of an algorithm making allocations in this setting is evaluated by its competitive ratio. The competitive ratio is the ratio of the objective value achieved by the algorithm to the offline optimum for all possible instances. We provide the formal definition below.
% \lily{I do not see any formal definition of the CR anywhere. You should add that.... otehrwise things like this do not make sense.}

\begin{definition}[Competitive Ratio]
Consider the set of decision vectors produced by an algorithm, ALG, as 
$\set{\bar{\vec{x}}_t}_{t=1}^T$ and the offline optimal decision vector that achieves $P^{\star}$ from Problem \eqref{eq:primal} as $\set{\vec{x}_t^{\star}}_{t=1}^T$. 
Then, ALG has a competitive ratio of $\alpha$ if
\[\alpha \leq \frac{\text{ALG}}{P^{\star}} = \frac{\sum_{t=1}^Tv_t\fp{\bar{\vec{x}}_t} - f\fp{\sum_{t=1}^T \bar{\vec{x}}_t}}{\sum_{t=1}^Tv_t\fp{\vec{x}_t^{\star}} - f\fp{\sum_{t=1}^T \vec{x}_t^{\star}}}\]
for all $\set{v_t}_{t=1}^T$.
\end{definition}
Note that $\alpha\in\bracket{0,1}$ and the closer to $1$, the better the algorithm.

\subsection{Primal-Dual Algorithm}
\label{sec:algorithm}
% \mitas{new section title upon giving algorithm a name}
% \lily{I actually think this algorithm should come in the section before. It fits better in the problem setup. This section should be able competitive ratio analysis and then sec 5 for designing/optimizing the surrogate function. You should make a subsection on the primal-dual algorithm. Then a subsection on the surrogate problem with the assumptions on the surrogate. Then give some context for the following two sections on CR analysis and surrogate design. Make the paper flow like a story. its out order right ow...}
% In this section, we present the primal-dual algorithm for the online optimization problem with procurement costs, and analyze its  competitive ratio.
%The proposed primal-dual algorithm in presented in Section \ref{sec:algorithm}, and the competitive ratio analysis is presented in Section \ref{sec:analysis}.
We now present the primal-dual algorithm for the online optimization problem with procurement costs. We first formulate the dual of \eqref{eq:primal} which is given by
\begin{equation}
\label{eq:dual}
D^{\star}:=\minimize_{\boldsymbol\lambda \succeq \vec{0},\vec{z}_t\succeq\vec{0}}~\sum_{t=1}^T \sum_{d=1}^D\max\!\set{\bracket{\vec{z}_t}_d-\bracket{\boldsymbol{\lambda}}_d, 0} - \sum_{t=1}^Tv_{t*}\fp{\vec{z}_t} + f^*\fp{\boldsymbol{\lambda}}.
\tag{D-1}
\end{equation}
From the computation of \eqref{eq:dual}---which is given in Appendix \ref{sec:computing_dual}---we develop an algorithm in which a greedy solution is obtained at time $t$ given previous decisions. The greedy solution solves a marginal optimization problem for the $t$-th time step considering that decisions for time steps $\bracket{t-1}$
have already been made. Let $\bar{\vec{x}}_i$ denote the decision made by an algorithm at time $i$. The greedy solution at time $t$ is the result of
\begin{equation}
\label{eq:marginal_opt}
\maximize_{\vec{0}\preceq \vec{x}_t\preceq\vec{1}}~
v_t\fp{\vec{x}_t} - f\fp{\sum_{i=1}^{t-1}\bar{\vec{x}}_i+\vec{x}_t} + f\fp{\sum_{i=1}^{t-1}\bar{\vec{x}}_i}.
\tag{M-1}
\end{equation}
The objective of \eqref{eq:marginal_opt} represents the gain in the objective of \eqref{eq:primal} at time $t$ if we make decision $\vec{x}_t$, since the decisions $\set{\bar{\vec{x}}_i}_{i=1}^{t-1}$ have already been made and cannot be changed. From Assumption \ref{assump:f}(1), we know that $f=f^{**}$, and from Assumption \ref{assump:v}(1), we know that $v_t=v_{t**}$, which allows us to re-write \eqref{eq:marginal_opt} as 
\begin{equation}
\label{eq:marginal_opt_dual}
\maximize_{\vec{0}\preceq \vec{x}_t\preceq\vec{1}}~\minimize_{\boldsymbol\lambda \succeq \vec{0},\vec{z}_t\succeq\vec{0}}~ \vec{z}_t^\top\vec{x}_t - v_{t*}\fp{\vec{z}_t} - \boldsymbol{\lambda}^\top\!\paren{\sum_{i=1}^{t-1}\bar{\vec{x}}_i+\vec{x}_t} + f^*\fp{\boldsymbol{\lambda}} + f\fp{\sum_{i=1}^{t-1}\bar{\vec{x}}_i}.
\tag{M-2}
\end{equation}

A greedy algorithm using this decision rule makes an allocation at time $t$ based on the incoming $v_t$, the previous decisions made, and $f$. In order to improve the performance of this algorithm with unknown future functions $v_t$, we ask the following question: \emph{can we design a surrogate function such that decisions made with respect to this function give a better competitive ratio for our original problem?}
Consider the following optimization problem, with the surrogate function denoted by $f_s$,
\begin{equation}
\label{eq:primal_fs} 
\begin{array}{rl}
\maximize& \sum_{t=1}^Tv_t\fp{\vec{x}_t} - f_s\fp{\sum_{t=1}^T \vec{x}_t}\\
\mbox{subject to}& \vec{0} \preceq \vec{x}_t \preceq \vec{1} \spforall t\in[T]\\
\end{array},\tag{P-2}
\end{equation}
where $\vec{x}_t \in [0,1]^D$ is the optimization variable and $v_t:\R^D_+\rar\R_+$ satisfies Assumption \ref{assump:v}. The only difference between Problem \eqref{eq:primal} and Problem \eqref{eq:primal_fs} is that $f$ has been replaced by $f_s$,
which satisfies the following assumptions.
\begin{assumption}
\label{assump:fs}
The function $f_s: \mathbb{R}_+^D\to \mathbb{R}_+$ satisfies the following:
\begin{enumerate}[itemsep=0pt]
\item $f_s$ is convex, differentiable, and closed.
\item $f_s$ is increasing; i.e., $\vec{u}\succeq\vec{v}$ implies $f_s\fp{\vec{u}}\geq f_s\fp{\vec{v}}$.
\item $f_s$ has an increasing gradient; i.e., $\vec{u}\succeq\vec{v}$ implies $\nabla f_s\fp{\vec{u}}\succeq 
\nabla f_s\fp{\vec{v}}$.
\item $f_s\fp{\vec{0}}=0$.
\item $f_s\fp{\vec{u}}\geq f\fp{\vec{u}}$ for all $\vec{0}\preceq\vec{u}\preceq T\vec{1}$.
\end{enumerate}
\end{assumption}
Assumptions \ref{assump:fs}(1)-(4) are identical to Assumptions \ref{assump:f}(1)-(4). Assumption \ref{assump:fs}(5) is designed to make sure the resulting algorithm makes allocations more cautiously than the greedy algorithm without a surrogate function in order to best handle the uncertainty of the future customers. A simple example to illustrate this intuition is provided in Section \ref{sec:numerical}. We discuss our choice of the surrogate function in more detail in Section \ref{sec:surrogate}.

Using the same strategy as above of writing the marginal optimization problem, now with respect to Problem \eqref{eq:primal_fs}, we can write the decision rule of Algorithm \ref{alg:sim_update}.
\begin{algorithm}
\SetKwInOut{Input}{Input}
\Input{$f_s:\R^D\rar\R$}
% $\bar{\vec{x}}_0 = \vec{0}$\;
% $\bar{\boldsymbol{\lambda}}_0 = \nabla f\fp{\bar{\vec{x}}_0}$\;
 \For{$t=1, \ldots, T$}{
  receive $v_t$\;
  $\paren{\bar{\boldsymbol{\lambda}}_t, \bar{\vec{z}}_t, \bar{\vec{x}}_t} = \argmin\limits_{\boldsymbol{\lambda}\succeq \vec{0},\vec{z}_t\succeq\vec{0}}~\max\limits_{\vec{0}\preceq \vec{x}_t\preceq\vec{1}}~ \vec{z}_t^\top\vec{x}_t - v_{t*}\fp{\vec{x}_t} - \boldsymbol{\lambda}^\top\!\paren{\sum_{i=1}^{t-1}\bar{\vec{x}}_i+\vec{x}_t} + f_s^*\fp{\boldsymbol{\lambda}}$\;
 }
 \caption{Simultaneous Update} \label{alg:sim_update}
\end{algorithm}

Line 3 in Algorithm \ref{alg:sim_update}, the main computational step of the algorithm, involves solving a (convex-concave) saddle-point problem. 
% While we do not discuss the details here, 
We point out that standard convex optimization methods (see, e.g., \cite{Bubeck2015}) 
% \maryam{I guess it's better to refer to algorithms that give rates/complexity specifically for saddle-point problems. I'll have to look, but maybe papers by Kilinc-Karzan? does Bubeck's book referred to here say something on this? *if there's no time, fine to leave Bubeck's book.}
can be used to solve this subproblem with desired accuracy, and the complexity analysis of these methods (number of iterations needed to reach  $\epsilon$-optimality) can be incorporated in the overall computational complexity analysis of our algorithm.
% We devise one such algorithm in Appendix \ref{sec:negotiation}. 
In Section \ref{sec:posted_pricing}, we discuss an algorithm that computes the primal and dual variables sequentially.

In the remainder of this section, let $\bar{\vec{x}}_t$ denote the decision vector at time $t$ given from Algorithm \ref{alg:sim_update} called with $f_s$. Algorithm \ref{alg:sim_update} called with $f_s$ ensures that at every time step $t$, 
\begin{equation}
\label{eq:primal_decision}
\bar{\vec{x}}_t = \1{\bar{\vec{z}}_t - \bar{\boldsymbol{\lambda}}_t \succeq \vec{0}},
\end{equation}
where $\bar{\boldsymbol{\lambda}}_t = \nabla f_s\fp{\sum_{i=1}^t\bar{\vec{x}}_i}$ and $\bar{\vec{z}}_t = \nabla v_t\fp{\bar{\vec{x}}_t}$.

The superscript notation of $s$---taken from \textit{surrogate}---denotes the objective of Problem \eqref{eq:primal} resulting from the decision vectors coming from Algorithm \ref{alg:sim_update} called with $f_s$. The primal objective is given by
\begin{equation}
\label{eq:Ps}
P^s := \sum_{t=1}^Tv_t\fp{\bar{\vec{x}}_t} - f\fp{\sum_{t=1}^T\bar{\vec{x}}_t},
\end{equation}
and the dual objective is given by
\begin{equation}
\label{eq:Ds}
D^s := \sum_{t=1}^T \sum_{d=1}^D\max\set{\bracket{\bar{\vec{z}}_t}_d-\bracket{\bar{\boldsymbol{\lambda}}_t}_d, 0} - \sum_{t=1}^Tv_{t*}\fp{\bar{\vec{z}}_t} + f^*\fp{\bar{\boldsymbol{\lambda}}_T}.  
\end{equation}
These equations are used in the analysis of Algorithm \ref{alg:sim_update} in Section \ref{sec:analysis}.

\section{Competitive Ratio Analysis for a Primal-Dual Algorithm}
\label{sec:analysis}
In this section, we bound the competitive ratio of Algorithm \ref{alg:sim_update} called with $f_s$ in Theorem 
\ref{thm:competitive_ratio}. In order to do this, we first show that Algorithm \ref{alg:sim_update} called with $f_s$ does not make a decision which causes the objective to become negative.
\begin{lemma}
\label{lemma:objective_positive}
If $f_s$ is convex and differentiable and $f_s\fp{\vec{0}}=0$, then 
\[\sum_{t=1}^Tv_t\fp{\bar{\vec{x}}_t} - f_s\fp{\sum_{t=1}^T\bar{\vec{x}}_t}\geq 0.\]
\end{lemma}
\begin{proof}
% Since $f_s\fp{\vec{0}}=0$, we have that 
We upper bound this expression by incorporating the decision rule of Algorithm \ref{alg:sim_update} called with $f_s$ as follows:
\begin{align*}
\sum_{t=1}^Tv_t\fp{\bar{\vec{x}}_t} - f_s\fp{\sum_{t=1}^T\bar{\vec{x}}_t}
&\overset{\text{(a)}}{\geq}\sum_{t=1}^T\nabla v_t\fp{\bar{\vec{x}}_t}^\top\vec{x}_t - f_s\fp{\sum_{t=1}^T\bar{\vec{x}}_t}
\\&\overset{\text{(b)}}{=}\sum_{t=1}^T\paren{\nabla v_t\fp{\bar{\vec{x}}_t}^\top\vec{x}_t - f_s\fp{\sum_{i=1}^t\bar{\vec{x}}_i} + f_s\fp{\sum_{i=1}^{t-1}\bar{\vec{x}}_i}} 
\\&\overset{\text{(c)}}{\geq} \sum_{t=1}^T\inprod{\nabla v_t\fp{\bar{\vec{x}}_t} - \nabla f_s\fp{\sum_{i=1}^t\bar{\vec{x}}_i}}{\bar{\vec{x}}_t}.
\end{align*}
Inequality (a) comes from concavity of $v_t$. Equality (b) comes from writing $f_s(\sum_{t=1}^T\bar{\vec{x}}_t)$ as a telescoping sum with the assumption that $f_s\fp{\vec{0}}=0$. Inequality (c) follows from convexity of $f_s$. 
Finally, the decision rule of Algorithm \ref{alg:sim_update}---i.e., $\bar{\vec{x}}_t =\1{\bar{\vec{z}}_t - \bar{\boldsymbol{\lambda}}_t \succeq \vec{0}}$---called with $f_s$ ensures that the inner product is always non-negative.
\end{proof}

Now, we bound the competitive ratio of Algorithm \ref{alg:sim_update} called with $f_s$.

\begin{theorem}
\label{thm:competitive_ratio}
Suppose that $f_s:\R^D\to \R$ satisfies Assumption \ref{assump:fs}. The competitive ratio of Algorithm \ref{alg:sim_update} (called with $f_s$) is bounded by
$1/\alpha_{f,f_s}$ where
\[\alpha_{f,f_s}:=\sup_{\vec{0}\preceq\vec{u}\preceq T\vec{1}}\frac{f^*\fp{\nabla f_s\fp{\vec{u}}}}{f_s\fp{\vec{u}} - f\fp{\vec{u}}}.\]
\end{theorem}
\begin{proof}
The general overview of the proof is as follows: 
% \lily{per comment *** above, reference the defintions that way its easy for the reader/reviewer to know what to refer back to for these objects like $D^s$.} 
writing $D^s$ \eqref{eq:Ds} in terms of $P^s$ \eqref{eq:Ps}, we bound the gap between $D^s$ and $P^s$. From here, we lower bound $D^s$ by $D^{\star}$ \eqref{eq:dual}, which in turn allows us to use weak duality to relate $D^{\star}$ and $P^{\star}$.

We start with writing $D^s$ in terms of $P^s$:
\begin{align*}
D^s &= \sum_{t=1}^T \sum_{d=1}^D\max\set{\bracket{\bar{\vec{z}}_t}_d-\bracket{\bar{\boldsymbol{\lambda}}_t}_d, 0} - \sum_{t=1}^Tv_{t*}\fp{\bar{\vec{z}}_t} + f^*\fp{\bar{\boldsymbol{\lambda}}_T}
\\&\overset{\text{(a)}}{=} \sum_{t=1}^T\paren{\bar{\vec{z}}_t-\bar{\boldsymbol{\lambda}}_t}^\top\bar{\vec{x}}_t - \sum_{t=1}^Tv_{t*}\fp{\bar{\vec{z}}_t} + f^*\fp{\bar{\boldsymbol{\lambda}}_T}
% = \sum_{t=1}^T\paren{\vec{c}_t-\nabla f_s\fp{\sum_{i=1}^t\bar{\vec{x}}_i}}^\top\bar{\vec{x}}_t + f^*\fp{\bar{\boldsymbol{\lambda}}_T}
\\&\overset{\text{(b)}}{=} \sum_{t=1}^T\nabla v_t\fp{\bar{\vec{x}}_t}^\top\bar{\vec{x}}_t - \sum_{t=1}^T\nabla f_s\fp{\sum_{i=1}^t\bar{\vec{x}}_i}^\top\bar{\vec{x}}_t - \sum_{t=1}^Tv_{t*}\fp{\bar{\vec{z}}_t} + f^*\fp{\bar{\boldsymbol{\lambda}}_T}.
\end{align*}
Equality (a) comes from the decision rule of Algorithm \ref{alg:sim_update} called with $f_s$, which ensures that $\bar{\vec{x}}_t = \1{\bar{\vec{z}}_t - \bar{\boldsymbol{\lambda}}_t \succeq \vec{0}}$. Equality (b) comes from replacing $\bar{\boldsymbol{\lambda}}_t$ with $\nabla f_s\fp{\sum_{i=1}^t\bar{\vec{x}}_i}$ and $\bar{\vec{z}}_t$ with $\nabla v_t\fp{\bar{\vec{x}}_t}$.
Now, we proceed to bound the duality gap between $D^s$ and $P^s$ by first observing the following relationship:
\begin{align*}
D^s &\overset{\text{(c)}}{\leq} \sum_{t=1}^T\nabla v_t\fp{\bar{\vec{x}}_t}^\top\bar{\vec{x}}_t - f_s\fp{\sum_{t=1}^T\bar{\vec{x}}_t} - \sum_{t=1}^Tv_{t*}\fp{\bar{\vec{z}}_t} + f^*\fp{\bar{\boldsymbol{\lambda}}_T}
\\&\overset{\text{(d)}}{=} \sum_{t=1}^T\nabla v_t\fp{\bar{\vec{x}}_t}^\top\bar{\vec{x}}_t - f_s\fp{\sum_{t=1}^T\bar{\vec{x}}_t} - \sum_{t=1}^T\paren{\nabla v_t\fp{\bar{\vec{x}}_t}^\top\bar{\vec{x}}_t-v_t\fp{\bar{\vec{x}}_t}} + f^*\fp{\bar{\boldsymbol{\lambda}}_T}
\\&= \sum_{t=1}^Tv_t\fp{\bar{\vec{x}}_t} - f_s\fp{\sum_{t=1}^T\bar{\vec{x}}_t} + f^*\fp{\bar{\boldsymbol{\lambda}}_T} + f\fp{\sum_{t=1}^T\bar{\vec{x}}_t} - f\fp{\sum_{t=1}^T\bar{\vec{x}}_t}
\\&\overset{\text{(e)}}{=} P^s - f_s\fp{\sum_{t=1}^T\bar{\vec{x}}_t} + f^*\fp{\bar{\boldsymbol{\lambda}}_T} + f\fp{\sum_{t=1}^T\bar{\vec{x}}_t}.
\end{align*}
Inequality (c) follows directly from convexity of $f_s$. Equality (d) comes from the concave Fenchel-Young inequality---i.e., equation \eqref{eq:fenchel-young-concave} with $g=v_t$ and $\vec{u}=\bar{\vec{x}}_t$. Equality (e) follows by substituting the definition of $P^s = \sum_{t=1}^Tv_t(\bar{\vec{x}}_t) - f(\sum_{t=1}^T\bar{\vec{x}}_t)$ where in the preceding equality we add and subtract $f(\sum_{t=1}^T\bar{\vec{x}}_t)$. We bound the gap between $D^s$ and $P^s$
as a multiplicative factor of $P^s$ in order to relate these quantities as a ratio:
\begin{align*}
\frac{f_s\fp{\sum_{t=1}^T\bar{\vec{x}}_t} - f^*\fp{\bar{\boldsymbol{\lambda}}_T} - f\fp{\sum_{t=1}^T\bar{\vec{x}}_t}}{P^s}
&\overset{\text{(f)}}{=} \frac{f_s\fp{\sum_{t=1}^T\bar{\vec{x}}_t} - f^*\fp{\nabla f_s\fp{\sum_{t=1}^T\bar{\vec{x}}_t}} - f\fp{\sum_{t=1}^T\bar{\vec{x}}_t}}{\sum_{t=1}^Tv_t\fp{\bar{\vec{x}}_t} - f\fp{\sum_{t=1}^T\bar{\vec{x}}_t}}
\\&\overset{\text{(g)}}{\geq} \frac{f_s\fp{\sum_{t=1}^T\bar{\vec{x}}_t} - f^*\fp{\nabla f_s\fp{\sum_{t=1}^T\bar{\vec{x}}_t}} - f\fp{\sum_{t=1}^T\bar{\vec{x}}_t}}{f_s\fp{\sum_{t=1}^T\bar{\vec{x}}_t} - f\fp{\sum_{t=1}^T\bar{\vec{x}}_t}}
\\&\overset{\text{(h)}}{\geq}\inf_{\vec{0}\preceq\vec{u}\preceq T\vec{1}} \frac{f_s\fp{\vec{u}} - f^*\fp{\nabla f_s\fp{\vec{u}}} - f\fp{\vec{u}}}{f_s\fp{\vec{u}} - f\fp{\vec{u}}}
=:\beta_{f,f_s}.
\end{align*}
In equality (f), we replace $\bar{\boldsymbol{\lambda}}_T$ with $\nabla f_s(\sum_{i=1}^T\bar{\vec{x}}_i)$. Inequality (g) follows from Lemma \ref{lemma:objective_positive}, and inequality (h) follows from observing that $\vec{0}\preceq \sum_{t=1}^T\bar{\vec{x}}_t\preceq T\vec{1}$. Hence,
\[\beta_{f,f_s}P^s\leq f_s\fp{\sum_{t=1}^T\bar{\vec{x}}_t} - f^*\fp{\bar{\boldsymbol{\lambda}}_T} - f\fp{\sum_{t=1}^T\bar{\vec{x}}_t}\leq P^s-D^s.\]
Define 
\begin{equation}
\label{eq:alpha_f_fs}
\alpha_{f,f_s} :=1-\beta_{f,f_s} = \sup_{\vec{0}\preceq\vec{u}\preceq T\vec{1}} \frac{f^*\fp{\nabla f_s\fp{\vec{u}}}}{f_s\fp{\vec{u}} - f\fp{\vec{u}}}.
\end{equation}
Assumption \ref{assump:fs}(5) ensures that $\alpha_{f,f_s}\geq 0$ which, in turn,  ensures that the competitive ratio is non-negative. We now lower bound $D^s$ by $D^{\star}$ and subsequently use weak duality to get that $D^s \geq D^{\star} \geq P^{\star}$. From Assumption \ref{assump:fs}(3), we know that $\nabla f_s(\sum_{i=1}^t\bar{\vec{x}}_i) \preceq \nabla f_s(\sum_{i=1}^T\bar{\vec{x}}_i)$ for all $t\in\bracket{T}$ since $\bar{\vec{x}}_i\succeq\vec{0}$ for all $i\in\bracket{T}$. This, in turn, implies that $\bar{\boldsymbol{\lambda}}_t \preceq \bar{\boldsymbol{\lambda}}_T$ for all $t\in\bracket{T}$ so that
\begin{align*}
D^s&=\sum_{t=1}^T \sum_{d=1}^D\max\set{\bracket{\bar{\vec{z}}_t}_d-\bracket{\bar{\boldsymbol{\lambda}}_t}_d, 0} - \sum_{t=1}^Tv_{t*}\fp{\bar{\vec{z}}_t} + f^*\fp{\bar{\boldsymbol{\lambda}}_T}
\\&\geq \sum_{t=1}^T \sum_{d=1}^D\max\set{\bracket{\bar{\vec{z}}_t}_d-\bracket{\bar{\boldsymbol{\lambda}}_T}_d, 0} - \sum_{t=1}^Tv_{t*}\fp{\bar{\vec{z}}_t} + f^*\fp{\bar{\boldsymbol{\lambda}}_T} \geq D^\star.
\end{align*}
Hence, $P^s - D^\star \geq P^s\beta_{f,f_s}$, and applying weak duality, we get that $P^s - P^\star \geq P^s\beta_{f,f_s}$. Rearranging this equation gives us the following:
\[\frac{P^s}{P^\star} \geq \frac{1}{1-\beta_{f,f_s}} = \frac{1}{\alpha_{f,f_s}}.\]
This concludes the proof.
\end{proof}

This theorem allows us to write the competitive ratio as the result of an optimization problem for a large class of $f$ and $f_s$. Our objective then becomes to design $f_s$ such that $\alpha_{f,f_s}$ is as small as possible, since this would in turn yield a stronger competitive ratio bound. We can then verify the following intuition: with the goal of increasing the denominator of \eqref{eq:alpha_f_fs}, we see that we must craft $f_s$ to be sufficiently larger than $f$ in order to make cautious allocations in the face of adversarial uncertainty. However, with the goal of decreasing the numerator of \eqref{eq:alpha_f_fs}, we must not design $f_s$ to be too large, otherwise the algorithm will make little to no allocation.
% Note that $1/\alpha_{f,f_s}$ depends on both the original $f$ and the surrogate $f_s$. 
In the next section, we discuss how to choose $f_s$ to optimize this ratio. 
% \lily{Give some interpretation of this result. Why is it meaningful? how does it compare? etc. then link it to the next section. For instance, the first sentence of the next section says that the surrogate is important. so explain here the roll that the surrogate function plays in plain language}
\section{Designing the Surrogate Function}
\label{sec:surrogate}
As the analysis in the preceding section shows, the choice of the surrogate function plays a crucial role in obtaining an improved
competitive ratio bound.
In this section, we propose techniques to design $f_s$ for particular classes of functions. In particular, in Section 5.1, we propose a technique for designing the surrogate of polynomial functions and we obtain the competitive ratio bound in this setting. In Section 5.2, we exploit quasiconvex optimization to design the surrogate function for a general $f$.

\subsection{Polynomial Function}
\label{sec:polynomial_surrogate}
We propose a design technique for a special class of $f$:  polynomials that satisfy Assumption \ref{assump:f}. We let $f_s\fp{\vec{u}}=\frac{1}{\rho}f\fp{\rho\vec{u}}, \rho>1$. Note that $\nabla f_s\fp{\vec{u}} = \nabla f\fp{\rho\vec{u}}$.  Due to Assumption \ref{assump:f}(3) (that the gradient of $f$ is increasing) looking ahead when making a decision forces us to use a larger gradient and be more careful in our allocation. The intuition here is to stay cautious because we make no assumptions on the arriving input.
Now, it suffices to determine $\rho$. This surrogate function was proposed in \cite{Chan2015}, but their analysis yielded a suboptimal choice of $\rho$.

To determine our choice of $\rho$, we start with Lemmas \ref{lemma:optimal_rho} and \ref{lemma:positive_numerator}.
Using these, Theorem \ref{thm:polynomial} shows that finding the optimal $\rho$ for a general class of polynomial functions comes back to solving the optimization problem in Lemma \ref{lemma:optimal_rho}.

\begin{lemma}
\label{lemma:optimal_rho}
For $\tau \geq 2$, the solution to
\[\argmin_{\rho}\paren{\tau-1}\frac{\rho^{\tau}}{\rho^{\tau-1}-1}\]
is $\rho^\star = \tau^{1/\paren{\tau-1}}$.
\end{lemma}
The proof is provided in Appendix \ref{sec:proof_of_optimal_rho}.

\begin{lemma}
\label{lemma:positive_numerator}
Given $\rho > 1$, for any $0\leq a \leq b$,
\[b\rho^{b-a}-a\frac{\rho^b-1}{\rho^a-1} \geq 0.\]
\end{lemma}
The proof is provided in Appendix \ref{sec:proof_of_positive_numerator}.

\begin{theorem}
\label{thm:polynomial}
Suppose that $\vec{u}\in \mathbb{R}^D$. %for some $D\in \N$.
For any $K\in \N$, suppose $f_K\fp{\vec{u}}=\sum_{k=1}^Kc_kg_k\fp{\vec{u}}$ is a convex function such that $c_k>0$ for each $k\in\bracket{K}$ and $g_k\fp{\vec{u}}=\prod_{i=1}^Du_i^{\tau_{ki}}$ where $\sum_{i=1}^D\tau_{ki}=\tau_k$, and $\tau_{ki}\in\R_+$ for all pairs $(k,i)\in \bracket{K}\times\bracket{D}$. Assume that $\tau:=\tau_{i^{\star}}\geq 2$ where $i^{\star}=\argmax_i \tau_i$. Then, choosing parameter $\rho$ as $\rho = \tau^{1/\paren{\tau-1}}$ guarantees a competitive ratio of at least 
$\tau^{-\tau/\paren{\tau-1}}$
for Algorithm \ref{alg:sim_update} called with $\frac{1}{\rho}f_K\fp{\rho\vec{u}}$. 
\end{theorem}
\begin{proof}
We first use induction to show that
\[\inf_{\rho>1}\sup_{\vec{0}\preceq\vec{u}\preceq T\vec{1}} \frac{f_K^*\fp{\nabla_{\rho\vec{u}} f_K\fp{\rho\vec{u}}}}{\frac{1}{\rho}f_K\fp{\rho\vec{u}}-f_K\fp{\vec{u}}}\leq \inf_{\rho>1} \paren{\tau-1}\frac{\rho^{\tau}}{\rho^{\tau-1}-1},\]
and then apply Lemma \ref{lemma:optimal_rho} to the optimization problem.
First note that for any $K\in \mathbb{N}$, the Fenchel-Young inequality holds with equality as described in Equation \eqref{eq:fenchel-young} in Section \ref{sec:preliminaries}. 
% (cf.~Section \ref{sec:preliminaries}.\lily{be more specific where in section 2}).
That is, \[f_K^*\fp{\nabla_{\rho\vec{u}} f_K\fp{\rho\vec{u}}} = \inprod{\nabla_{\rho\vec{u}}f_{K}\fp{\rho\vec{u}}}{\rho\vec{u}}-f_{K}\fp{\rho\vec{u}}.\] We now begin the inductive argument on $K$. For $K=1$, 
\[f_1\fp{\vec{u}} = c_1\prod_{i=1}^Du_i^{\tau_{1,i}},\]
where $\sum_{i=1}^D\tau_{1,i}=\tau_{1}\geq 2$ and $\tau_{1,i}$ is non-negative for all $i$. 

Using the definition of $f_1$, we have
% \lily{$f$ or $f_1$?}
\begin{align*}
\inf_{\rho>1}\sup_{\vec{0}\preceq\vec{u}\preceq T\vec{1}} \frac{f_1^*\fp{\nabla_{\rho\vec{u}} f_1\fp{\rho\vec{u}}}}{\frac{1}{\rho}f_1\fp{\rho\vec{u}}-f_1\fp{\vec{u}}}
&\overset{\text{(a)}}{=} \inf_{\rho>1}\sup_{\vec{0}\preceq\vec{u}\preceq T\vec{1}} \frac{\inprod{\nabla_{\rho\vec{u}}f_1\fp{\rho\vec{u}}}{\rho\vec{u}}-f_1\fp{\rho\vec{u}}}{\frac{1}{\rho}f_1\fp{\rho\vec{u}}-f_1\fp{\vec{u}}}\\
&\overset{\text{(b)}}{=} \inf_{\rho>1}\sup_{\vec{0}\preceq\vec{u}\preceq T\vec{1}} \frac{\rho^{\tau_1}\paren{\tau_1-1}f_1\fp{\vec{u}}}{\paren{\rho^{\tau_1-1}-1}f_1\fp{\vec{u}}}\\
&\overset{\text{(c)}}{=} \inf_{\rho>1} \paren{\tau_1-1}\frac{\rho^{\tau_1}}{\rho^{\tau_1-1}-1}.
\end{align*}
Equality (a) comes from the Fenchel-Young inequality holding with equality. Equality (b) comes from the following:
\begin{align*}
f_1\fp{\rho\vec{u}}&=\rho^{\tau_1}f_1\fp{\vec{u}},\\ \inprod{\nabla_{\rho\vec{u}}f_1\fp{\rho\vec{u}}}{\rho\vec{u}}&=\rho^{\tau_1}\tau_1f_1\fp{\vec{u}}.
\end{align*}
Equality (c) comes from removing $f_1\fp{\vec{u}}$ from the numerator and denominator, thus eliminating any dependence of $\vec{u}$ in the optimization problem. This concludes the proof for $K=1$.

Suppose that the result holds for $K-1\in \mathbb{N}$. We argue the result for $K\in \mathbb{N}$.
For notational convenience, we define
\begin{align*}
a_k&:=c_k\rho^{\tau_k-\tau_K}\paren{\tau_k-1},\\
b_k&:=c_k\paren{\tau_K-1}\paren{\frac{\rho^{\tau_k-1}-1}{\rho^{\tau_K-1}-1}},\\
h_k\fp{\vec{u}}&:=\paren{g_k\fp{\vec{u}}}^{-1}.
\end{align*}
Without loss of generality, let $\tau_1\geq \cdots\geq \tau_K$ where $\tau_1\geq 2$. Our strategy is to show that removing $c_Kg_K\fp{\vec{u}}$ upper bounds the optimization problem. We first isolate the the $c_Kg_K\fp{\vec{u}}$ term. Then, we show that keeping this term reduces the objective. We then finish the claim with the inductive hypothesis.

We begin with the following:
\begin{align*}
\frac{\inprod{\nabla_{\rho\vec{u}}f_{K}\fp{\rho\vec{u}}}{\rho\vec{u}}-f_{K}\fp{\rho\vec{u}}}{\frac{1}{\rho}f_{K}\fp{\rho\vec{u}}-f_{K}\fp{\vec{u}}}
&\overset{\text{(d)}}{=}\frac{\sum_{k=1}^{K}c_k\rho^{\tau_k}\paren{\tau_k-1}g_k\fp{\vec{u}}}{\sum_{k=1}^{K}\paren{\rho^{\tau_k-1}-1}g_k\fp{\vec{u}}}
\\&=\frac{\sum_{k=1}^{K-1}c_k\rho^{\tau_k}\paren{\tau_k-1}g_k\fp{\vec{u}} + c_K\rho^{\tau_{K}}\paren{\tau_{K}-1}g_K\fp{\vec{u}}}{\sum_{k=1}^{K-1}c_k\paren{\rho^{\tau_k-1}-1}g_k\fp{\vec{u}} + c_K\paren{\rho^{\tau_{K}-1}-1}g_K\fp{\vec{u}}}
\\&\overset{\text{(e)}}{=}\frac{\rho^{\tau_{K}}}{\rho^{\tau_{K}-1}-1}\paren{\frac{h_K\fp{\vec{u}}\paren{\sum_{k=1}^{K-1}a_kg_k\fp{\vec{u}}} + c_{K}\paren{\tau_{K}-1}}{h_K\fp{\vec{u}}\paren{\sum_{k=1}^{K-1}c_k\paren{\frac{\rho^{\tau_k-1}-1}{\rho^{\tau_{K}-1}-1}}g_k\fp{\vec{u}}} + c_{K}}}
\\&\overset{\text{(f)}}{=}\frac{\rho^{\tau_{K}}}{\rho^{\tau_{K}-1}-1}\paren{\paren{\tau_{K}-1}+ \frac{h_K\fp{\vec{u}}\sum_{k=1}^{K-1}\paren{a_k-b_k}g_k\fp{\vec{u}}}{h_K\fp{\vec{u}}\paren{\sum_{k=1}^{K-1}c_k\paren{\frac{\rho^{\tau_k-1}-1}{\rho^{\tau_{K}-1}-1}}g_k\fp{\vec{u}}} + c_{K}}}.
\end{align*}
Equality (d) comes from the following:
\begin{align*}
f_K\fp{\rho\vec{u}}&=\sum_{k=1}^Kc_k\rho^{\tau_k}g_k\fp{\vec{u}},\\ \inprod{\nabla_{\rho\vec{u}}f_K\fp{\rho\vec{u}}}{\rho\vec{u}}&=\sum_{k=1}^K\rho^{\tau_k}\tau_kg_k\fp{\vec{u}}.
\end{align*}
Equality (e) comes from factoring out $\rho^{\tau_K}g_K\fp{\vec{u}}$ from the numerator and $\paren{\rho^{\tau_K-1}-1}g_K\fp{\vec{u}}$ from the denominator. Equality (f) comes from rearranging the fraction inside the parentheses by bringing $\paren{\tau_K-1}$ out front.

We have successfully isolated $c_K$ in the denominator of the fraction. Since $c_{K}> 0$, adding it to the denominator shrinks the fraction inside the parentheses since the numerator is positive, which we know from Lemma \ref{lemma:positive_numerator}. Indeed, Lemma \ref{lemma:positive_numerator} shows that for all $k\in\bracket{K}$,
\[a_k-b_k=c_k\paren{\rho^{\tau_k-\tau_{K}}\paren{\tau_k-1}-\paren{\tau_{K}-1}\paren{\frac{\rho^{\tau_k-1}-1}{\rho^{\tau_{K}-1}-1}}} \geq 0.\]
Now, we have
\begin{align*}
\frac{\inprod{\nabla_{\rho\vec{u}}f_{K}\fp{\rho\vec{u}}}{\rho\vec{u}}-f_{K}\fp{\rho\vec{u}}}{\frac{1}{\rho}f_{K}\fp{\rho\vec{u}}-f_{K}\fp{\vec{u}}}
&\overset{\text{(g)}}{\leq}\frac{\rho^{\tau_{K}}}{\rho^{\tau_{K}-1}-1}\paren{\paren{\tau_{K}-1}+ \frac{h_K\fp{\vec{u}}\sum_{k=1}^{K-1}\paren{a_k-b_k}g_k\fp{\vec{u}}}{h_K\fp{\vec{u}}\sum_{k=1}^{K-1}c_k\paren{\frac{\rho^{\tau_k-1}-1}{\rho^{\tau_{K}-1}-1}}g_k\fp{\vec{u}}}}
\\&\overset{\text{(h)}}{=}\frac{\rho^{\tau_{K}}}{\rho^{\tau_{K}-1}-1}\paren{\paren{\tau_{K}-1}+ \frac{\sum_{k=1}^{K-1}\paren{a_k-b_k}g_k\fp{\vec{u}}}{\sum_{k=1}^{K-1}c_k\paren{\frac{\rho^{\tau_k-1}-1}{\rho^{\tau_{K}-1}-1}}g_k\fp{\vec{u}}}}
\\&\overset{\text{(i)}}{=}\frac{\sum_{k=1}^{K-1}\rho^{\tau_k}\paren{\tau_k-1}g_k\fp{\vec{u}}}{\sum_{k=1}^{K-1}\paren{\rho^{\tau_k-1}-1}g_k\fp{\vec{u}}}
\\&=\frac{\inprod{\nabla_{\rho\vec{u}}f_{K-1}\fp{\rho\vec{u}}}{\rho\vec{u}}-f_{K-1}\fp{\rho\vec{u}}}{\frac{1}{\rho}f_{K-1}\fp{\rho\vec{u}}-f_{K-1}\fp{\vec{u}}}.
\end{align*}
% \lily{break this up. add more context for what you are doign. I would be extremely annoyed as a reviewer}
Inequality (g) comes from removing $c_{K}$ from the denominator. Equality (h) comes from removing $h_K\fp{\vec{u}}$ from the numerator and denominator of the fraction inside the parentheses. Equality (i) comes from combining the expression back into a single fraction. We now finish the claim with the following:
%Now, we can incorporate the Fenchel-Young inequality, which holds at equality, and have the following:
\begin{align*}
\inf_{\rho>1}\sup_{\vec{0}\preceq\vec{u}\preceq T\vec{1}} \frac{f_{K}^*\fp{\nabla_{\rho\vec{u}} f_{K}\fp{\rho\vec{u}}}}{\frac{1}{\rho}f_{K}\fp{\rho\vec{u}}-f_{K}\fp{\vec{u}}}
&\overset{\text{(j)}}{=}\inf_{\rho>1}\sup_{\vec{0}\preceq\vec{u}\preceq T\vec{1}}\frac{\inprod{\nabla_{\rho\vec{u}}f_{K}\fp{\rho\vec{u}}}{\rho\vec{u}}-f_{K}\fp{\rho\vec{u}}}{\frac{1}{\rho}f_{K}\fp{\rho\vec{u}}-f_{K}\fp{\vec{u}}}
\\&\leq \inf_{\rho>1}\sup_{\vec{0}\preceq\vec{u}\preceq T\vec{1}}\frac{\inprod{\nabla_{\rho\vec{u}}f_{K-1}\fp{\rho\vec{u}}}{\rho\vec{u}}-f_{K-1}\fp{\rho\vec{u}}}{\frac{1}{\rho}f_{K-1}\fp{\rho\vec{u}}-f_{K-1}\fp{\vec{u}}}
\\&\overset{\text{(k)}}{=}\inf_{\rho>1}\sup_{\vec{0}\preceq\vec{u}\preceq T\vec{1}}\frac{f_{K-1}^*\fp{\nabla_{\rho\vec{u}} f_{K-1}\fp{\rho\vec{u}}}}{\frac{1}{\rho}f_{K-1}\fp{\rho\vec{u}}-f_{K-1}\fp{\vec{u}}}
\\&\overset{\text{(l)}}{\leq}\inf_{\rho>1}\paren{\tau_1-1}\frac{\rho^{\tau_1}}{\rho^{\tau_1-1}-1}.
\end{align*}
Equality (j) and equality (k) come from the Fenchel-Young inequality, which holds at equality. Inequality (l) comes from the inductive hypothesis. 

We now apply Lemma \ref{lemma:optimal_rho} to solve
\[\argmin_{\rho>1}\paren{\tau_1-1}\frac{\rho^{\tau_1}}{\rho^{\tau_1-1}-1}=\tau_1^{1/\paren{\tau_1-1}}.\]
Plugging this choice of $\rho$ back into the objective gives us $\tau_1^{\tau_1/\paren{\tau_1-1}}$ which concludes the proof.
\end{proof}

\paragraph{Comparison to \cite{Chan2015}.} 
% We use the notation in the statement of Theorem \ref{thm:polynomial}. 
Chan et al. \cite{Chan2015} approach a similar optimization problem, but exploit their additional assumptions on the procurement cost function which essentially restricts their class to polynomials. They choose their design parameter to be $\rho = \lambda^{1/\paren{\lambda-1}}$ where $\lambda$ is defined as the smallest cumulative degree of a term in $f$; i.e., $\lambda:=\tau_{j^{\star}}$ where $j^{\star}=\argmin_j\tau_j$. Chan et al. \cite{Chan2015} are interested in the asymptotic behavior of the competitive ratio in terms of $\tau$, and both their choice of $\rho$ and our choice of $\rho$ give the same $\bigO{\tau}$ competitive ratio bound\footnote{In their work, \cite{Chan2015} define the competitive ratio to be the inverse of ours, and so to avoid confusion in case the reader refers to their work, we compare their result with ours according to their definition of competitive ratio.}. However, we achieve a more refined competitive ratio bound with our choice of $\rho=\tau^{1/\paren{\tau-1}}$.
% , due to the constant factor in the result of \cite{Chan2015}, which is absorbed in the asymptotic notation. \blue{Competitive ratio analysis is not typically studied in terms of the asymptotic function of a problem parameter (see Section 2.2.1 in \cite{Mehta2013} for a survey of results), and especially in this setting, an asymptotic result in terms of the largest power does not fully capture the performance of the algorithm.}

\subsection{General Case}
\label{sec:surrogate_general}
In this section, we propose a design approach for a general procurement cost function. We show that the algorithm metric we aim to optimize is a quasiconvex function of $f_s$, the surrogate function we are aiming to design. Therefore, the search over an appropriate family of $f_s$ can be carried out by quasiconvex optimization. 
Note that while the approach is general, solving the problem computationally requires discretizing the variable $\vec{u}\in\R_+^D$, and thus this method is suitable for cases where $D$ is small. 

\begin{theorem}
\label{thm:quasiconvex}
Let $f\fp{\vec{u}}=\sum_{n=1}^Ng_n\fp{\vec{u}}$ where $g_n$ satisfies Assumption \ref{assump:f}
% is convex and has an increasing gradient 
for all $n\in\bracket{N}$. Let $\vec{a}\in\R^N$, where $\vec{a}\succeq\vec{1}$, and $f_s\fp{\vec{u}}=\sum_{n=1}^Na_ng_n\fp{\vec{u}}$. Consider a discretization of the set $\setcond{\vec{u}}{\vec{0}\preceq\vec{u}\preceq T\vec{1}}$ and denote the points in this discretized set as $\U$. The following problem
\begin{equation}
\label{eq:quasiconvex} 
\minimize_{\vec{a}\succeq\vec{1}}~\max_{\vec{u}\in\U}~\frac{f^*\fp{\nabla f_s\fp{\vec{u}}}}{f_s\fp{\vec{u}}-f\fp{\vec{u}}}
\tag{Q-1}
\end{equation}
can be solved as a quasiconvex optimization problem.
\end{theorem}
\begin{proof}
In order to show that Problem \eqref{eq:quasiconvex} is a quasiconvex optimization problem, we must verify that the constraints are convex and the objective is quasiconvex.
It suffices to show that
\[\max_{\vec{u}\in\U}~\frac{f^*\fp{\nabla f_s\fp{\vec{u}}}}{f_s\fp{\vec{u}}-f\fp{\vec{u}}}\]
is a quasiconvex function in $\vec{a}$. Since a non-negative weighted maximum of quasiconvex functions is also quasiconvex, it suffices to show that $\frac{f^*\fp{\nabla f_s\fp{\vec{u}}}}{f_s\fp{\vec{u}}-f\fp{\vec{u}}}$ is quasiconvex in $\vec{a}$ for a fixed $\vec{u}$. We can directly apply the definition of quasiconvexity. Let $S_{\alpha}\fp{f_s}$ be the sub-level sets of $f_s$ for $\vec{a}\in\R^N$. We have the following:
\begin{align*}
S_{\alpha}\fp{f_s}&=\setcond{\vec{a}\succeq\vec{1}}{\frac{f^*\fp{\nabla f_s\fp{\vec{u}}}}{f_s\fp{\vec{u}}-f\fp{\vec{u}}}\leq\alpha}
\\&= \setcond{\vec{a}\succeq\vec{1}}{f^*\fp{\nabla f_s\fp{\vec{u}}}\leq\alpha\paren{f_s\fp{\vec{u}}-f\fp{\vec{u}}}}.
\end{align*}
For a fixed value of $\vec{u}$, $\bracket{\nabla f_s\fp{\vec{u}}}_d$ is linear in $\vec{a}$ for all $d$ and since $f^*$ is always convex, composing a convex function with a linear function of $\vec{a}$ is convex in $\vec{a}$. Finally, since $f_s\fp{\vec{u}}$ is linear in $\vec{a}$, the constraints of $S_{\alpha}\fp{f_s}$ are convex, and thus $S_{\alpha}\fp{f_s}$ is a convex set.
\end{proof}

Since Problem \eqref{eq:quasiconvex} is a quasiconvex optimization problem from Theorem \ref{thm:quasiconvex}, we can solve it by a sequence of convex feasibility problems. For a fixed $\alpha\in\left[1,\infty\right)$, we consider the following feasibility problem:
\begin{equation}
\label{eq:feasibility}
\begin{array}{rl}
\find & \vec{a} \\ 
\subjectto & f^*\fp{\nabla f_s\fp{\vec{u}}} \leq \alpha\paren{f_s\fp{\vec{u}}-f\fp{\vec{u}}} \spforall \vec{u}\in\U\\
&\vec{a}\succeq\vec{1}.
\end{array}
\tag{Q-2}
\end{equation}
We now perform a binary search on $\alpha$ to find the smallest $\alpha$, up to $\epsilon$ precision, such that Problem \eqref{eq:feasibility} is feasible. We write pseudocode for this procedure in Appendix \ref{sec:pseudocode_quasiconvex}.

% \paragraph{Comparison to \cite{Eghbali2016}.} 

\section{Numerical Examples}
\label{sec:numerical}

In this section, we illustrate the performance of our algorithm for specific procurement cost functions. In our first example, we use a simple procurement cost function in order to demonstrate the need for a surrogate function 
% that is different from the original procurement cost function 
when calling Algorithm \ref{alg:sim_update}. In our second example, we consider a non-separable polynomial procurement cost function and compare the performance of Algorithm \ref{alg:sim_update} for different surrogate function design techniques.

Consider the procurement cost function $f\fp{u}=u^2$, where $u\in\R_+$. This numerical example shows the necessity of a surrogate function, and how running Algorithm \ref{alg:sim_update} with the original procurement cost function has a competitive ratio of $0$ asymptotically. We show this by crafting a particular arrival instance in which we highlight the weakness of not using a surrogate function. The intuition is that not using a surrogate function allows the decision making to be excessively greedy, in that the algorithm does not caution itself from accumulating a large procurement cost for minimal revenue. This instance, described below, forces the algorithm to accumulate a large procurement cost before seeing higher valued arrivals which come soon after. In this instance, the incoming valuations are linear, and so we have $v_t\fp{x_t}=c_tx_t$. We have $c_t = 2t$. Assume that $T$ is divisible by $2$. From the decision rule of Algorithm \ref{alg:sim_update} called with $f_s$, i.e., $\bar{x}_t=\1{c_t-f_s'\fp{\sum_{i=1}^t\bar{x}_i}}$, calling Algorithm \ref{alg:sim_update} with $f_s\fp{u}=u^2$ leads to an allocation of $x_t=1$ for all $t$ which gives a cumulative reward of $T$. The optimal allocation is one that sets $x_t=1$ for all $t>\frac{T}{2}$ and yields an objective of $\frac{1}{2}\paren{T^2+T}$. This yields a ratio that tends to $0$ as $T$ becomes large. For each of our design techniques, the surrogate function is the same and is $f_s\fp{u}=2u^2$. Calling Algorithm \ref{alg:sim_update} with $f_s\fp{u}=2u^2$ leads to an allocation of $x_t=0.5$ for all $t$ which gives an objective of $\frac{T^2}{4}+\frac{T}{2}$. This yields a ratio that tends to $\frac{1}{2}$ as $T$ becomes large. This example then shows that even for a simple procurement cost function, not using a surrogate function may lead to a competitive ratio that tends to $0$ as $T$ becomes large.
% Figure \ref{fig:simple_alg_illustration} compares the objectives up to time $t$, i.e., $\sum_{i=1}^tv_i\fp{\bar{x}_i}-f\fp{\sum_{i=1}^t\bar{x}_i}$, of Algorithm \ref{alg:sim_update} called with a surrogate function, and without. For the surrogate function, we label it $f_{\text{poly}}$, but note that this surrogate function is identical when generated by either of our techniques or the technique proposed in \cite{Chan2015}. The surrogate function is $2u^2$. We have the label $f$ to represent Algorithm \ref{alg:sim_update} being called with the original production cost function. Note that since the objective of Algorithm \ref{alg:sim_update} called with a surrogate function is quadratic in time, and the objective of Algorithm \ref{alg:sim_update} called without a surrogate function is linear in time, the competitive ratio must approach $0$ as $T$ becomes large.

% \begin{figure}[h]
% \centering
% \includegraphics[width=0.6\textwidth]{arxiv/simple_plot_objective_vs_t.pdf}
% \caption{A plot comparing the the objectives up to time $t$ of Algorithm \ref{alg:sim_update} called with a surrogate function versus Algorithm \ref{alg:sim_update} called without one for the instance $v_t\fp{x_t}=c_tx_t$ where $c_t=2t$ and $f\fp{u}=u^2$.}
% \label{fig:simple_alg_illustration}
% \end{figure}

Now, consider the procurement cost function $f\fp{\vec{u}}=u_1^4 + \paren{u_1+u_2}^2$ where $\vec{u}\in\R^2_+$. Figure \ref{fig:general_case} shows the shape of the surrogate function using the design techniques from Sections \ref{sec:polynomial_surrogate} and \ref{sec:surrogate_general} respectively. For $f_s$ from Section \ref{sec:polynomial_surrogate}, we use the surrogate function $f_s\fp{\vec{u}}=\frac{1}{\rho}f\fp{\rho\vec{u}}$ and with Theorem \ref{thm:polynomial}, we choose $\rho=4^{1/3}$. This means that $f_s\fp{\vec{u}} = 4u_1^4+4^{1/3}\paren{u_1+u_2}^2$. This choice of $\rho$ then gives a competitive ratio bound of $4^{-4/3}\approx 0.1575$. For $f_s$ from Section \ref{sec:surrogate_general}, we use surrogate function $f_s\fp{\vec{u}} = a_1u_1^4 + a_2\paren{u_1+u_2}^2$ from Theorem \ref{thm:quasiconvex}. To solve Problem \eqref{eq:quasiconvex}, we set $T=10$ and have $100$ points per square unit in the discretization; i.e., \[\mathcal{U}=\setcond{\vec{u}}{u_i\in\set{0, 0.1, 0.2, \dots, 9.9, 10}~\forall i\in\set{1,2}}.\]
This achieves the competitive ratio bound of approximately $0.1577$ with $a_1 \approx 3.791$ and $a_2 \approx 2.386$. The surrogate function from Section \ref{sec:surrogate_general} allows for an additional design parameter which allows us to achieve a slightly better competitive ratio bound than the surrogate function from Section \ref{sec:polynomial_surrogate}. However, the technique from Section \ref{sec:surrogate_general} has a much higher computational cost due to numerically solving the quasiconvex optimization problem in Problem \eqref{eq:quasiconvex}.
\begin{figure}[h]
\centering
\includegraphics[width=0.4\textwidth]{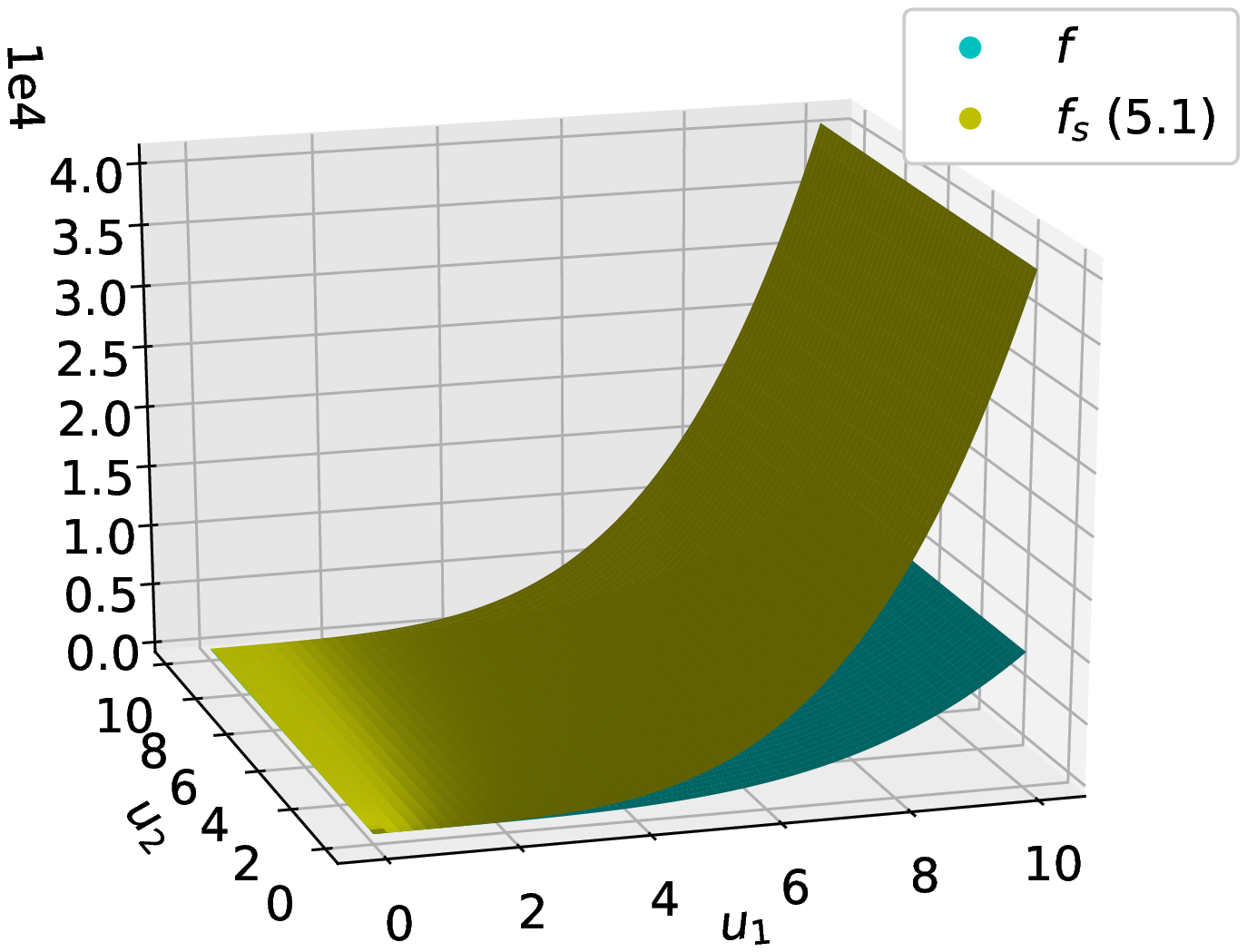}
\includegraphics[width=0.4\textwidth]{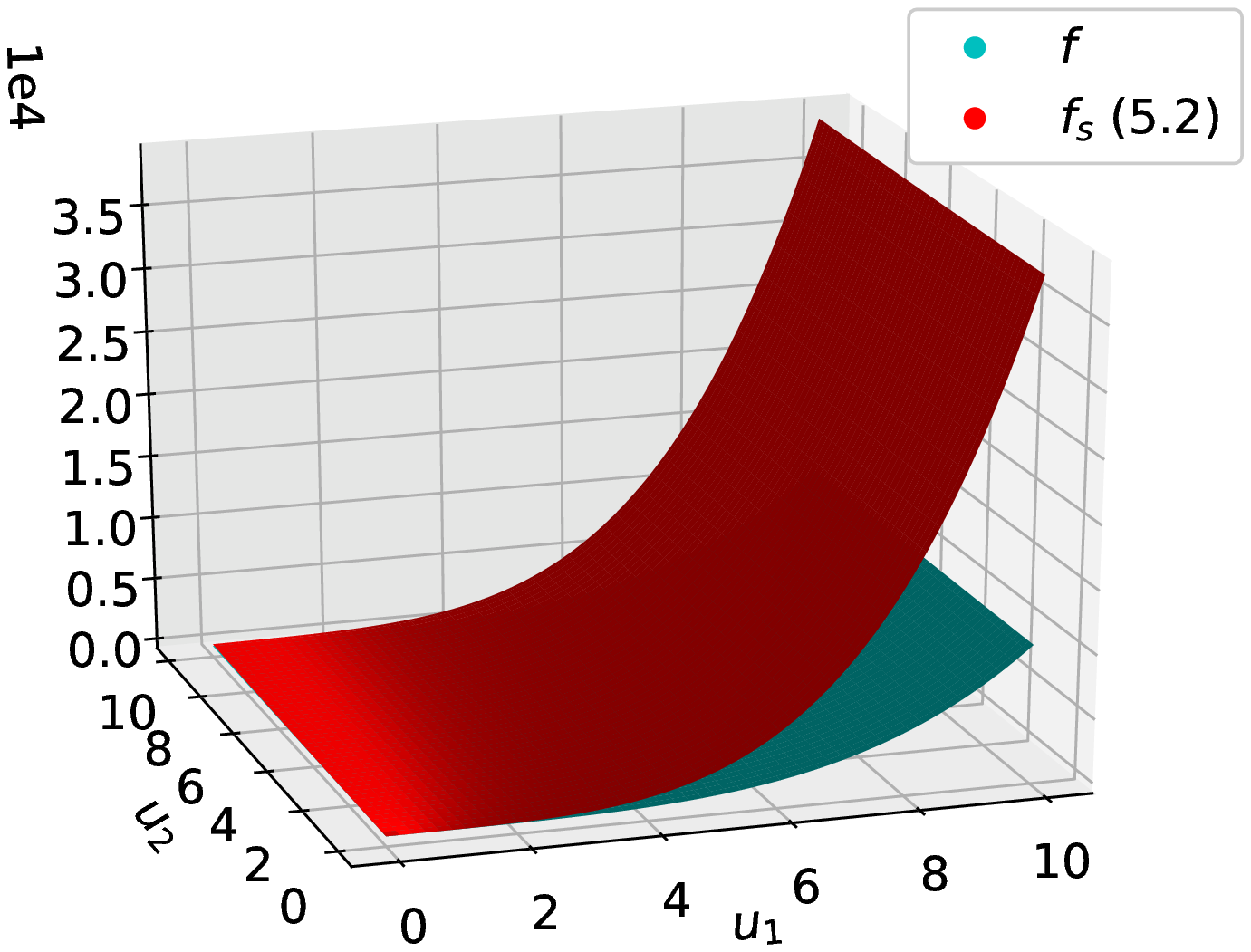}
\caption{Problem \eqref{eq:quasiconvex} applied to $f\fp{\vec{u}}=u_1^4 + \paren{u_1+u_2}^2$. $f_s$ (5.1) represents the surrogate function from using the technique in Section \ref{sec:polynomial_surrogate}. $f_s$ (5.2) represents the surrogate function from using the technique in Section \ref{sec:surrogate_general}.}
% \Description{A plot of the original function and the surrogate function for a non-separable 2-dimensional production cost function.}
\label{fig:general_case}
\end{figure}
Figure \ref{fig:original_alg_illustration} compares the cumulative objective values up to time $t$, i.e., $\sum_{i=1}^tv_i\fp{\bar{\vec{x}}_i}-f\fp{\sum_{i=1}^t\bar{\vec{x}}_i}$, of Algorithm \ref{alg:sim_update} called with different surrogate functions. 
% ($\bar{\vec{x}}_i$ comes from Algorithm \ref{alg:sim_update} called with different surrogate functions)
For the surrogate functions, we have the label $f$ representing the surrogate function equal to original production cost function, and so Algorithm \ref{alg:sim_update} is called with $u_1^4 + \paren{u_1+u_2}^2$. We have the label $f_{\text{poly}}$ representing the surrogate function from using the technique in Section \ref{sec:polynomial_surrogate}, so Algorithm \ref{alg:sim_update} is called with $4u_1^4+4^{1/3}\paren{u_1+u_2}^2$. We have the label $f_{\text{design}}$ representing the surrogate function from using the technique in Section \ref{sec:surrogate_general}, so Algorithm \ref{alg:sim_update} is called with $a_1u_1^4 + a_2\paren{u_1+u_2}^2$ where $a_1 \approx 3.791$ and $a_2 \approx 2.386$. Finally, we have the label $f_{\text{chk}}$ representing the surrogate function from using the technique in \cite{Chan2015}, so Algorithm \ref{alg:sim_update} is called with $8u_1^3 + 2\paren{u_1+u_2}^2$. The online arrivals are generated by reasoning about instances that would be adversarial for Algorithm \ref{alg:sim_update} called with the original procurement cost function. The weakness in calling Algorithm \ref{alg:sim_update} called with $f$, i.e, not using a surrogate function, is that the decisions are made too greedily, in that the algorithm does not caution itself from accumulating a large procurement cost for minimal revenue. This instance is thus generated from having online arrivals which force the algorithm to amass a large procurement cost before seeing higher valued arrivals which it can no longer take. In this instance, the incoming valuations are linear, i.e., $v_t\fp{\vec{x}_t}=\vec{c}_t^\top\vec{x}_t$, where
\[\vec{c}_t=\begin{cases}\nabla f\fp{t\cdot\vec{1}}&\text{if $t$ is odd}\\\nabla f\fp{2t\cdot\vec{1}}&\text{if $t$ is even}\end{cases}.\]

\begin{figure}[h]
\centering
\includegraphics[width=0.6\textwidth]{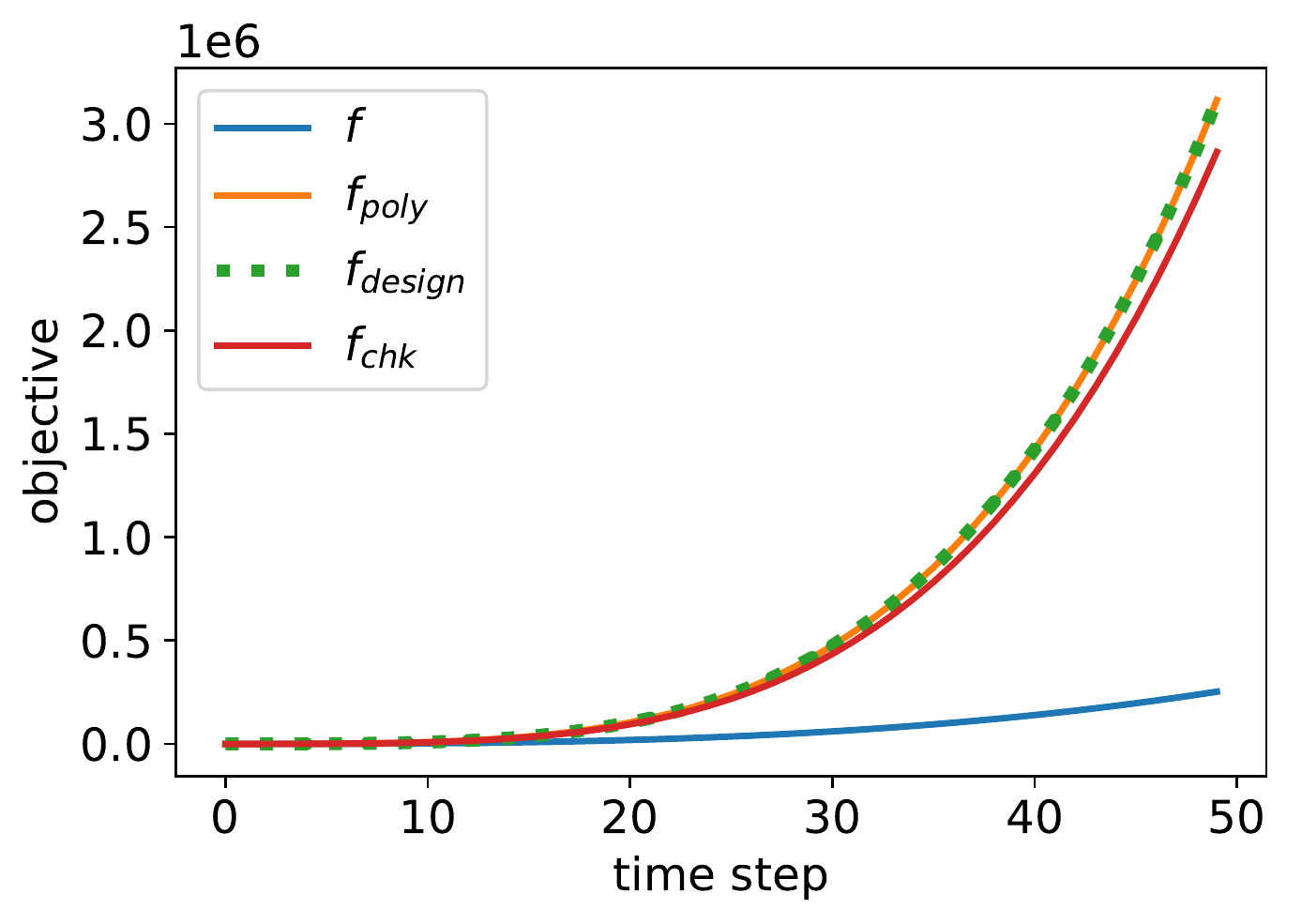}
\caption{A plot comparing the the objectives up to time $t$ of Algorithm \ref{alg:sim_update} called with different surrogate functions. For the surrogate functions, we have $f$ representing the surrogate function equal to $f$, $f_{\text{poly}}$ representing the surrogate function from using the technique in Section \ref{sec:polynomial_surrogate}, $f_{\text{design}}$ representing the surrogate function from using the technique in Section \ref{sec:surrogate_general}, and $f_{\text{chk}}$ using the technique in \cite{Chan2015}.}
\label{fig:original_alg_illustration}
\end{figure}

\section{Posted Pricing Mechanisms}
\label{sec:posted_pricing}
In this section, we propose Algorithm \ref{alg:seq_update_offset} which is a primal-dual algorithm that computes the primal and dual variables sequentially, unlike Algorithm \ref{alg:sim_update} which computes the primal and dual variables simultaneously as the solution to the saddle-point problem in equation \eqref{eq:marginal_opt_dual}. Algorithm \ref{alg:seq_update_offset} is much more computationally efficient, and also possess an economic interpretation of \textit{incentive compatibility} as defined in Definition \ref{def:incentive_compatible}.

\begin{definition}[Incentive compatibility]
\label{def:incentive_compatible}
An online algorithm for problem \eqref{eq:primal} is called incentive compatible when each customer maximizes their utility by being truthful, i.e., each customer reports and acts according to their true beliefs.
\end{definition}

% Incentive compatibility is desirable in applications where the privacy of $v_t$ is important. 
Algorithm \ref{alg:seq_update_offset} is called a posted pricing mechanism, as defined in  Definition \ref{def:posted_pricing}, and immediately satisfies \textit{incentive compatibility}.

\begin{definition}[Posted Pricing Mechanism]
\label{def:posted_pricing}
An online algorithm is a posted pricing mechanism when the seller posts item prices and allows the arriving customer to choose their desired bundle of items given the prices.
\end{definition}

The interpretation here, is that upon arrival, the customer chooses the allocation which maximizes their utility, and this would be identical to the allocation that the seller would assign had the user reported their true valuation function. From the notation of Algorithm \ref{alg:seq_update_offset}, the dual variable, $\bar{\boldsymbol{\lambda}}_t$, represents a price that is revealed at each time step, \textit{before} the customer arrives, and then the allocation for this arriving customer is then determined by this price. The posted price at time step $t$, therefore, does not depend on $v_t$, and so the arriving agent does not need to reveal it. A posted pricing mechanism is therefore desirable in applications where the privacy of $v_t$ is important. 

% This is desirable in applications where the privacy of $\vec{c}_t$ is important and holds the additional property of incentive compatibility, meaning that if the arriving agent has uses $\vec{c}_t$ to make the allocation, then the agent should be truthful in order to maximize her utility \cite{Balcan2008, Chawla2010}.
% and holds the additional property of incentive compatibility, meaning that if the arriving agent has the ability to misreport $v_t$, the agent should instead be truthful in order to obtain an allocation that maximizes her utility \cite{Balcan2008, Chawla2010}.

% \subsection{Algorithm}
\begin{algorithm}
\SetKwInOut{Input}{Input}
\Input{$f:\R^D\rar\R, \vec{v}_{\text{offset}}\in\R^D_+$}
% $\bar{\vec{x}}_0 = \vec{0}$\;
% $\bar{\boldsymbol{\lambda}}_0 = \nabla f\fp{\bar{\vec{x}}_0+\vec{v}_{\text{offset}}}$\;
 \For{$t=1 \dots T$}{
  $\bar{\boldsymbol{\lambda}}_t = \nabla f\fp{\sum_{i=1}^{t-1} \bar{\vec{x}}_i + \vec{v}_{\text{offset}}}$\;
%   $\bar{\vec{x}}_t = \1{\vec{c}_t-\bar{\boldsymbol{\lambda}}_t \succeq \vec{0}}$\;
  $\bar{\vec{x}}_t = \argmax_{\vec{0}\preceq\vec{x}_t\preceq\vec{1}}~v_t\fp{\vec{x}_t}-\bar{\boldsymbol{\lambda}}_t^\top\vec{x}_t$
 }
 \caption{Sequential Update with Offset}
 \label{alg:seq_update_offset}
\end{algorithm}
% \maryam{The naming of simultaneous versus sequential updates comes from Reza's paper, cite and mention.}
We propose the primal-dual algorithm in Algorithm \ref{alg:seq_update_offset}. Here, in comparison to Algorithm \ref{alg:sim_update}, $\bar{\boldsymbol{\lambda}}_t$ is being used to set the threshold at time $t$, independent of the allocation made at time $t$. Thus, the value of $\bar{\boldsymbol{\lambda}}_t$ does not require solving a saddle-point problem. Furthermore, in comparison to Algorithm \ref{alg:sim_update}, in addition to passing in the function, $f$, as an argument, we pass in an offset vector, $\vec{v}_{\text{offset}}$, to Algorithm \ref{alg:seq_update_offset} which allows us to additively control the threshold. The naming of both Algorithm \ref{alg:sim_update} as Simultaneous Update and Algorithm \ref{alg:seq_update_offset} as Sequential Update to distinguish between how the primal and dual variables are computed, come from \cite{Eghbali2016}.

\subsection{Analysis without Offset}
In this subsection, we analyze the competitive ratio of Algorithm \ref{alg:seq_update_offset} called with $\vec{v}_{\text{offset}}=\vec{0}$ and $f_s$ satisfying Assumption \ref{assump:fs}. This ensures that at every time step $t$, $\bar{\vec{x}}_t = \1{\bar{\vec{z}}_t - \bar{\boldsymbol{\lambda}}_t \succeq \vec{0}}$, where $\bar{\boldsymbol{\lambda}}_t = \nabla f_s\fp{\sum_{i=1}^{t-1}\bar{\vec{x}}_i}$ and $\vec{z}_t=\nabla v_t\fp{\bar{\vec{x}}_t}$ from Lemma \ref{lemma:v_optimality}. Now, we bound the competitive ratio of Algorithm \ref{alg:seq_update_offset}.
\begin{theorem}
\label{thm:competitive_ratio_seq_0}
Let $f_s$ satisfy Assumption \ref{assump:fs}. The competitive ratio of Algorithm \ref{alg:seq_update_offset} called with $f_s$ and $\vec{v}_{\text{offset}}=\vec{0}$ is bounded by $1/\alpha_{f,f_s}$ where
\[\alpha_{f,f_s}:=\sup_{\vec{0}\preceq\vec{u}\preceq T\vec{1}}\frac{f^*\fp{\nabla f_s\fp{\vec{u}}}}{f_s\fp{\vec{u}} - f\fp{\vec{u}}-\vec{1}^\top\!\paren{\nabla f_s\fp{\vec{u}}-\nabla f_s\fp{\vec{0}}}}.\]
\end{theorem}
This proof is very similar to that of Theorem \ref{thm:competitive_ratio} and so the proof is provided in Appendix \ref{sec:proof_of_competitive_ratio_seq_0}.

\paragraph{Designing the general surrogate function} In similar vein to Section \ref{sec:surrogate_general}, we now propose a design technique for the surrogate function, $f_s$ to be used in Algorithm \ref{alg:seq_update_offset} based on Theorem \ref{thm:competitive_ratio_seq_0}.

\begin{theorem}
\label{thm:quasiconvex_seq_0}
Let $f\fp{\vec{u}}=\sum_{n=1}^Ng_n\fp{\vec{u}}$ where $g_n$ satisfies Assumption \ref{assump:f}
% is convex and has an increasing gradient 
for all $n\in\bracket{N}$. Let $\vec{a}\in\R^N$, where $\vec{a}\succeq\vec{1}$, and $f_s\fp{\vec{u}}=\sum_{n=1}^Na_ng_n\fp{\vec{u}}$. Consider a discretization of the set $\setcond{\vec{u}}{\vec{0}\preceq\vec{u}\preceq T\vec{1}}$ and denote the points in this discretized set as $\U$. The following problem
\begin{equation}
\label{eq:quasiconvex_seq_0} 
\minimize_{\vec{a}\succeq\vec{1}}~\max_{\vec{u}\in\U}~\frac{f^*\fp{\nabla f_s\fp{\vec{u}}}}{f_s\fp{\vec{u}}-f\fp{\vec{u}}-\vec{1}^\top\!\paren{\nabla f_s\fp{\vec{u}}-\nabla f_s\fp{\vec{0}}}}
\tag{Q-2}
\end{equation}
can be solved as a quasiconvex optimization problem.
\end{theorem}
This proof is very similar to that of Theorem \ref{thm:quasiconvex} and so the proof is provided in Appendix \ref{sec:proof_of_quasiconvex_seq_0}.

\subsection{Analysis with Offset}
In this section, we show that posting a \textit{more cautious} price, i.e., setting a larger threshold due to the uncertainty from the allocation, allows for a clean analysis of the competitive ratio of Algorithm \ref{alg:seq_update_offset}. We term a larger price as \textit{more cautious}, since an allocation is not made unless the larger threshold is reached, implying a larger degree of caution 
% in terms of allocating 
for the current time step. This larger threshold comes from the assumption that the gradient of $f_s$ is increasing, and so adding a non-negative offset to the argument increases $\nabla f_s$.

% We consider an additional assumption For the rest of this subsection, we consider Algorithm \ref{alg:seq_update_offset} called with $\vec{v}_{\text{offset}}=\vec{1}$ and $f_s$
In this subsection, we analyze Algorithm \ref{alg:seq_update_offset} called with $f_s$ satisfying Assumption \ref{assump:fs_seq} and $\vec{v}_{\text{offset}}=\vec{1}$. This ensures that at every time step $t$, $\bar{\vec{x}}_t = \1{\bar{\vec{z}}_t - \bar{\boldsymbol{\lambda}}_t \succeq \vec{0}}$, where $\bar{\boldsymbol{\lambda}}_t = \nabla f_s\fp{\sum_{i=1}^{t-1}\bar{\vec{x}}_i+\vec{1}}$ and $\bar{\vec{z}}_t=v_t\fp{\bar{\vec{x}}_t}$ from Lemma \ref{lemma:v_optimality}. 

We now consider the following assumptions on $f_s$.
\begin{assumption}
\label{assump:fs_seq}
The function $f_s: \mathbb{R}_+^D\to \mathbb{R}_+$ satisfies the following:
\begin{enumerate}
\item $f_s$ is convex, differentiable, and closed.
\item $f_s$ is increasing; i.e., $\vec{u}\succeq\vec{v}$ implies $f_s\fp{\vec{u}}\geq f_s\fp{\vec{v}}$.
\item $f_s$ has an increasing gradient; i.e., $\vec{u}\succeq\vec{v}$ implies $\nabla f_s\fp{\vec{u}}\succeq f_s\fp{\vec{v}}$.
\item $f_s\fp{\vec{0}}=0$.
%$f_s$ is normalized; i.e., it is zero at the origin.
\item $f_s\fp{\vec{u}}\geq f\fp{\vec{u}}$ for all $\vec{0}\preceq\vec{u}\preceq \paren{T-1}\vec{1}$.
\item $f_s\fp{\vec{a}}-f\fp{\vec{a}}\leq f_s\fp{\vec{b}}-f\fp{\vec{b}}$ if $\vec{0}\preceq\vec{a}\preceq\vec{b}$.
\end{enumerate}
\end{assumption}
Assumptions \ref{assump:fs_seq}(1)-(4) are identical to Assumptions \ref{assump:fs}(1)-(4).

We now bound the competitive ratio of Algorithm \ref{alg:seq_update_offset}.
\begin{theorem}
\label{thm:competitive_ratio_seq_1}
Let $f_s$ satisfy Assumption \ref{assump:fs_seq}. The competitive ratio of Algorithm \ref{alg:seq_update_offset} called with $f_s$ and $\vec{v}_{\text{offset}}=\vec{1}$ is bounded by
$1/\alpha_{f,f_s}$ where
\[\alpha_{f,f_s}:=\sup_{\vec{0}\preceq\vec{u}\preceq \paren{T-1}\vec{1}}\frac{f^*\fp{\nabla f_s\fp{\vec{u}+\vec{1}}}}{f_s\fp{\vec{u}} - f\fp{\vec{u}}}.\]
\end{theorem}
This proof is very similar to that of Theorem \ref{thm:competitive_ratio} and so the proof is provided in Appendix \ref{sec:proof_of_competitive_ratio_seq_1}.

\paragraph{Designing the general surrogate function} In similar vein to Section \ref{sec:surrogate_general}, we now propose a design technique for the surrogate function, $f_s$ to be used in Algorithm \ref{alg:seq_update_offset} based on Theorem \ref{thm:competitive_ratio_seq_1}.

\begin{theorem}
\label{thm:quasiconvex_seq_1}
Let $f\fp{\vec{u}}=\sum_{n=1}^Ng_n\fp{\vec{u}}$ where $g_n$ satisfies Assumption \ref{assump:f}
% is convex and has an increasing gradient 
for all $n\in\bracket{N}$. Let $\vec{a}\in\R^N$, where $\vec{a}\succeq\vec{1}$, and $f_s\fp{\vec{u}}=\sum_{n=1}^Na_ng_n\fp{\vec{u}}$. Consider a discretization of the set $\setcond{\vec{u}}{\vec{0}\preceq\vec{u}\preceq \paren{T-1}\vec{1}}$ and denote the points in this discretized set as $\U$. The following problem
\begin{equation}
\label{eq:quasiconvex_seq_1} 
\minimize_{\vec{a}\succeq\vec{1}}~\max_{\vec{u}\in\U}~\frac{f^*\fp{\nabla f_s\fp{\vec{u}+\vec{1}}}}{f_s\fp{\vec{u}}-f\fp{\vec{u}}}
\tag{Q-3}
\end{equation}
can be solved as a quasiconvex optimization problem.
\end{theorem}
This proof is very similar to that of Theorem \ref{thm:quasiconvex} and so the proof is provided in Appendix \ref{sec:proof_of_quasiconvex_seq_1}.

\section{Related Work}
\label{sec:related}

In this section, we review further related work at the intersection of online matching and combinatorial auctions.

% \subsection{Online Matching} 
% \todo{Adwords, packing/covering, reza, regularization (FTRL, MWU see game theory papers for citations for these)}
\paragraph{Online Bipartite Matching.} Online bipartite matching \cite{Karp1990,Kalyanasundaram2000,Devanur2013,Kesselheim2013} is a classical problem that has been studied and reintroduced for many applications. Recently, the natural application of internet ad placement has caused a resurgence of online bipartite matching and its generalizations through the problem Adwords \cite{Mehta2007,Devanur2009}. In the Adwords problem, a search engine is trying to maximize revenue from a set of budget-constrained advertisers, who bid on queries arriving online. This problem was generalized to allow the revenue to be the sum of a concave function of the budget spent for each advertiser \cite{Devanur2012}. All of the aforementioned problems have a separable cumulative budget constraint that must be satisfied, and so the algorithm techniques of choosing the allocation as a function of the budget is not applicable for our problem. 
% The problem was further generalized to online conic optimization in which the objective is a concave function and constraint sets and linear maps arrive online \cite{Eghbali2016}. \todo{add to this paragraph the fundamental differences between online matching and our set up}

\paragraph{Primal-Dual Algorithms.} State-of-the-art techniques for Adwords, its generalizations, and related problems have been primal-dual algorithms \cite{Buchbinder2007,Buchbinder2009}. A primal-dual algorithm uses the dual problem formulation, and updates the dual variables in order to determine the values of the primal variables. The advantages of primal-dual algorithms are two-fold. Firstly, the analysis for competitive ratio of a primal-dual algorithm then decomposes into writing the dual objective of the algorithm in terms of the primal objective, since weak duality can then be used to connect the two (see the opening paragraph of our proof of Theorem \ref{thm:competitive_ratio}). Secondly, the dual variable may have a meaningful interpretation in how to determine the primal variable. We adapt the intuition for primal and dual variables from problems of profit maximization \cite{Balcan2008,Chawla2010}. Although these problems are different from our framework, \cite{Balcan2008} considers limited or unlimited supply of resources and \cite{Chawla2010} considers customers arriving from a known distribution, the interpretations of the primal and dual variables are key in developing our posted pricing mechanism in Section \ref{sec:posted_pricing}. In both  Algorithm \ref{alg:sim_update} and Algorithm \ref{alg:seq_update_offset}, our allocation rules comes naturally from realizing that the payment obtained must be greater than the additional production cost. The dual variable can then be interpreted as the price offered to the incoming buyer, as further discussed in Section \ref{sec:posted_pricing}.
% , i.e., mechanisms in which the seller posts resource prices to each customer and lets the customer choose their utility maximizing bundle. 
% Posted pricing mechanisms are inherently incentive compatible, meaning that it is in the customer's best interest to behave according to their true preferences, and are  
% \omid{Probably, you've already defined incentive-compatibility and posted price mechanisms earlier in the paper, no need to explain it in so much detail here}

This powerful tool of duality is best seen in online covering and packing problems \cite{Chan2015, Azar2016}. The offline covering problem can be written as:
\[\minimize_{\vec{x}\in\R^n_+}~f\fp{\vec{x}}\spaces\subjectto~ A\vec{x}\succeq\vec{1},\]
where $f$ is a non-negative increasing convex cost function and $A$ is an $m\times n$ matrix with non-negative entries. In the online problem, rows of $A$ come online and a feasible assignment $\vec{x}$ must be maintained at all times where $\vec{x}$ may only increase. The offline covering problem can be written as:
\[\maximize_{\vec{y}\in\R^m_+}~\sum_j y_j-f^*\fp{A^\top\vec{y}},\]
and in the online setting, columns of $A^\top$ arrive online upon which $y_t$ must be assigned. The packing problem is dual to the covering problem as the $j$-th entry of $\vec{y}$ corresponds to the $j$-th row of $A$. In the works of \cite{Chan2015} and \cite{Azar2016}, the authors use this duality to analyze similar algorithms proposed for each problem. In fact, the bulk of the results in \cite{Chan2015} are focused on the covering and packing problems, upon which the authors then adapt their results to the online resource allocation problem in Section 5 of their work. In this paper, we obtain stronger results for the online resource allocation problem by studying the problem directly rather than trying to adapt results from the related problem of online packing.

We share a similar perspective in this work with \cite{Eghbali2016}. The authors there study a generalization of Adwords in which the objective is a concave function and constraint sets and linear maps arrive online. There, the authors propose a convex optimization problem to design a surrogate function in order to improve the competitive ratio; however, the problem studied in \cite{Eghbali2016} is different from ours in the following ways: (1) the data coming online in \cite{Eghbali2016} is linear, whereas in our setting the payment functions arriving online are generally concave, and (2) the objective of the offline optimization problem in \cite{Eghbali2016} is a coupled term between allocations at different rounds, but in our objective, in equation \eqref{eq:primal}, we have a sum over decoupled terms representing the cumulative payment, as well as a coupled term in the procurement cost function. Since these key differences do not allow our problem to be mapped to that in \cite{Eghbali2016}, we must develop separate surrogate function design techniques based on the competitive ratio analysis for our problem.

% The goal of the online covering problem is to minimize a non-negative increasing convex cost function, $f\fp{\vec{x}}$ subject to a linear covering constraint of the form $A\vec{x}\geq 1$, where $A$ is an $m\times n$ matrix, where the rows of $A$ come online and a feasible assignment $\vec{x}$ must be maintained at all times where $\vec{x}$ may only increase \omid{So many 'where's}. The online packing problem on the other hand, simply has an objective of maximizing $\sum y_j-f^*\fp{A^\top\vec{y}}$. Immediately, we see that the packing problem is dual to the covering problem, and so these problems are often studied together, with duality being used to analyze algorithms proposed for each problem. In fact, the bulk of the results in \cite{Chan2015} are focused on the covering and packing problems, upon which the authors then adapt their results to the online resource allocation problem in Section 5 of their work \omid{Write the online packing and covering problems in display math format instead of describing it, and mention why they are dual of each other, for example $y_j$ corresponds to the $j$-th row of $A$ arriving online}.

\paragraph{Arrival Models.} Most of the online optimization problems analyzed with respect to competitive ratio are studied under three arrival models: (1) the worst-case/adversarial model, with no assumptions on how the requests arrive, (2) the random order model, where the set of requests is arbitrary but the order of arrival is uniformly random, and (3) the independently and identically distributed (IID) model, where the requests are IID samples from an underlying distribution. For a more in depth survey, see Section 2.2 in \cite{Mehta2013}. Our setting is that of the worst-case model. The key approach to problems in the worst-case model is for the decision maker to apply a greedy algorithm which maximizes a function of how much revenue can be immediately gained versus how much revenue may be achieved later. In doing so, the decision maker must be cautious in spending budget or accumulating a resource which may be better consumed in the future. This decision making strategy connects loosely to the ideas of regularization for online optimization problems in the regret metric as seen in classical algorithms such as follow the regularized leader \cite{McMahan2011} and multiplicative weights \cite{Littlestone1994}. A key difference however from the regret setting to the competitive ratio setting, is that in the regret setting, regularization aims to keep the gap between the current and previous decision small, whereas in the competitive ratio setting, the regularization is used to make cautious decisions in order to protect resources which may obtain more value if used in future allocations. 

In the random order model, the typical approach is to have an exploration period, where the decision maker learns about the distribution of the arriving requests, followed by an exploitation period in which the decision maker uses this knowledge to maximize their revenue. This is most clearly seen in the classical secretary problem described in \cite{Chow1964} in which a set of candidates arrive one by one for an open job position, and the manager must hire or reject the candidate before interviewing future candidates. Adwords is studied in the random order model in \cite{Devanur2009} and the algorithm proposed uses the same technique of initial exploration, in which the bids on the first few queries are used to learn weights on the bidders used to select the allocation, and an exploitation period, in which these weights are applied to future queries to make the assignment. Similar strategies are used for generalizations of Adwords such as online linear programming \cite{Agrawal2014, Agrawal2014Devanur} and profit maximization subject to convex costs \cite{Gupta2018}. The key difference between the random order model and our setting of the worst case model, is that previous customers tell us nothing about future customers, and so we forgo learning about our customers, and focus solely on cautiously allocating our resources.

% \subsection{Combinatorial Auctions}
% \todo{blum, related work section of huang kim, first paragraph of blum paper on limited and unlimited supply}
\paragraph{Online Combinatorial Auctions.} In many related works, our problem of online resource allocation has been titled \textit{online combinatorial auctions}. Online combinatorial auctions have been studied in the setting with fixed resource capacities, i.e., there is a hard budget constraint for each resource \cite{Blumrosen2007,Balcan2008,Chakraborty2013} and in the setting with unlimited resource supplies, in which additional resources can be acquired at no cost \cite{Balcan2005,Balcan2008}. Our setting falls in between these; resources can be acquired or developed following a procurement cost. This problem was proposed by \cite{Blum2011} for separable procurement cost functions in the worst-case arrival model. \cite{Blum2011} devised a posted pricing mechanism, in which customers wanting to purchase the $k$-th copy of any item would be charged a price equal to the procurement cost of the $2k$-th copy of that item. \cite{Huang2018} build on this result by characterizing the competitive ratio of optimal algorithms in this setting for a wide range of separable procurement costs as the solution to a differential equation. Our framework looks to generalize this setting by considering non-separable production cost functions. Additionally, we bring an optimization viewpoint to this setting in which we use (quasi-)convex optimization to design the best surrogate function, rather than restricting ourselves to a small function class as do these papers. 
% Both papers focus heavily on posted pricing mechanisms primarily due to the property of incentive compatibility, and so our results in Section \ref{sec:posted_pricing} fit nicely into this community.
% Reasonable assumptions on this procurement cost function allow us to apply the state-of-the-art technique of primal-dual algorithms  to this problem. 
% Our posted pricing mechanism fits in this community. 
% \omid{The whole related work section should be written in separate paragraphs, where each paragraph discusses the prior work on a certain related problem, and then, compares and contrasts our framework and techniques with theirs.}
\section{Conclusion \& Future Directions}
In this paper, we studied the broad online optimization framework of online resource allocation with procurement costs. We analyzed the competitive ratio for a primal-dual algorithm and showed how we can design a surrogate function in order to improve the competitive ratio. We proposed two techniques to design or shape the surrogate function. The first technique, discussed in Section \ref{sec:polynomial_surrogate}, addressed the case of polynomial cost functions and 
determined a closed-form choice for the scalar design parameter, that guarantees a competitive ratio of at least $\tau^{-\tau/\paren{\tau-1}}$ where $\tau$ is the largest cumulative degree of a single term in the polynomial. This bound is optimal from a result in \cite{Huang2018} (Theorem 10).
The second technique, discussed in Section \ref{sec:surrogate_general}, considered a general class of procurement cost functions and relied on an optimization problem which is quasiconvex in the design parameters to determine a surrogate function. This allowed us to further improve the competitive ratio at a higher computational cost.
%to optimize the competitive ratio. 
In Section \ref{sec:numerical} we investigated the surrogate function arising from each design technique for numerical examples.

%Our immediate next step is to generalize Theorem \ref{thm:quasiconvex} by removing the restriction that each term of the production cost function with a design coefficient, i.e., $g_n$, must be convex. 

As a future direction, we aim to generalize Theorem \ref{thm:quasiconvex} to allow
%us to consider 
a much larger class of functions for the design of the surrogate. 
%The current statement of Theorem \ref{thm:quasiconvex} does not specify a unique choice of $g_n$ functions to be chosen,
We will also investigate which choice of $g_n$ would lead to optimal smoothing for a certain class of $f$.
Future steps also include a modified analysis that would allow more flexibility in $f$ but make more assumptions on the arriving inputs. 
Additionally, practically motivated assumptions on the structure of the incoming payment functions might lead to competitive ratio results for Algorithm \ref{alg:sim_update} that will not approach zero if $f_s$ is close to $f$. 
Furthermore, the different assumptions on the input order such as the random order model may be more suitable for certain applications, and competitive analysis in this regime has yet to be studied for this problem. In addition, different assumptions on the procurement cost function may be better 
suited for applications where the procurement cost functions satisfy gradient increasing, i.e., Assumption \ref{assump:f}(3) (continuous supermodular functions), but are not necessarily convex \cite{Sadeghi2019, Sadeghi2019Eghbali}.
%More specifically, the current analysis for the competitive ratio does not take into account the case where $f_s=f$; i.e., when there is no surrogate function and the algorithm used is algorithm \ref{alg:sim_update}. 
%This is because having $f_s=f$ would cause the value of $\alpha_{f,f_s}$ to become large which would imply that the competitive ratio would approach $0$. The competitive ratio analysis is performed in this manner due to there being no assumptions on the input. 
%However, this analysis also does not consider the structure of the incoming payment functions, so 
% In this same vein, we can explore how more general payment functions, such as concave payment functions rather than linear, affect the analysis.

% \begin{acks}
% \todo{start}
% \end{acks}

\bibliographystyle{alpha}
\bibliography{refs}

\appendix

\section{Computation of Dual Problem}
\label{sec:computing_dual}

\begin{align*}
P^\star&=\maximize_{\vec{0}\preceq \vec{x}_t\preceq\vec{1}}~\sum_{t=1}^Tv_t\fp{\vec{x}_t} - f\fp{\sum_{t=1}^T\vec{x}_t}
\\&= \maximize_{\vec{0}\preceq \vec{x}_t\preceq\vec{1}, \vec{y},\vec{w}_t}~ \sum_{t=1}^Tv_t\fp{\vec{w}_t} - f\fp{\vec{y}}
\\& \spaces ~\text{subject to }~\vec{y}\succeq \sum_{t=1}^T\vec{x}_t
\\& \spaces \spaces \spaces \spaces \spaces \spaces~ \vec{w_t}\preceq \vec{x}_t \spforall t\in\bracket{T}
\\&= \maximize_{\vec{0}\preceq \vec{x}_t\preceq\vec{1}, \vec{y},\vec{w_t}}~\minimize_{\boldsymbol\lambda \succeq \vec{0},\vec{z}_t\succeq\vec{0}}~\sum_{t=1}^T\paren{v_t\fp{\vec{w}_t} - f\fp{\vec{y}} + \vec{z}^\top\!\paren{\vec{x}_t-\vec{w}_t}} - f\fp{\vec{y}} + \boldsymbol\lambda^\top\!\paren{\vec{y}-\sum_{t=1}^T\vec{x}_t}
\\&\leq \minimize_{\boldsymbol\lambda \succeq \vec{0},\vec{z}_t\succeq\vec{0}}~ \paren{\maximize_{\vec{0}\preceq\vec{x}_t\preceq\vec{1}}~\sum_{t=1}^T\paren{\vec{z}_t-\boldsymbol{\lambda}}^\top\vec{x}_t} - \sum_{t=1}^T\paren{\minimize_{\vec{w}_t}~\vec{z}_t^\top\vec{w}_t-v_t\fp{\vec{w}_t}} + \paren{\maximize_{\vec{y}}~\boldsymbol{\lambda}^\top\vec{y}-f\fp{\vec{y}}}
\\&= \minimize_{\boldsymbol\lambda \succeq \vec{0},\vec{z}_t\succeq\vec{0}}~\sum_{t=1}^T\sum_{d=1}^D \max\!\set{\bracket{\vec{z}_t}_d-\bracket{\boldsymbol{\lambda}}_d, 0} - \sum_{t=1}^Tv_{t*}\fp{\vec{z}_t} + f^*\fp{\boldsymbol{\lambda}} = D^{\star}
\end{align*}

The inequality comes from weak duality. From the KKT (Karush-Kuhn-Tucker) conditions, we know that solving for the optimal value of $\vec{y}$ following the inequality comes from taking the gradient with respect to $\vec{y}$, and setting this equal to $0$. Solving this equation gives us $\boldsymbol{\lambda}^{\star} = \nabla f\fp{\vec{y}^\star}$. Since we know from our construction of $\vec{y}$ that $\vec{y}^\star = \sum_{t=1}^T\vec{x}_t^\star$, we can conclude that $\boldsymbol{\lambda}^\star = \nabla f\fp{\sum_{t=1}^T\vec{x}_t^\star}$. Similarly, solving for $\vec{w}_t$ gives us $\vec{z}_t^\star=\nabla v_t\fp{\vec{w}_t^\star}$ and from our construction of $\vec{w}_t$ that $\vec{w}_t^\star=\vec{x}_t^\star$, we can conclude that $\vec{z}_t^\star=\nabla v_t\fp{\vec{x}_t^\star}$. Our optimal values are then
\begin{align*}
\vec{x}_t^\star &= \1{\vec{z}_t^\star - \boldsymbol{\lambda}^\star \succeq \vec{0}},\\
\boldsymbol{\lambda}^\star &= \nabla f\fp{\sum_{t=1}^T \vec{x}_t^\star},\\
\vec{z}_t^\star &=\nabla v_t\fp{\vec{x}_t^\star}.
\end{align*}

\section{Missing Proofs in Section \ref{sec:surrogate}}

\subsection{Proof of Lemma \ref{lemma:optimal_rho}}
\label{sec:proof_of_optimal_rho}
Since $\paren{\tau-1}$ is a constant, we disregard it from the objective of the optimization problem, when we solve. We solve the optimization problem for $\rho$ by computing the first derivative,
\[\frac{d}{d\rho}\paren{\frac{\rho^{\tau}}{\rho^{\tau-1}-1}}=\tau\paren{\rho^{\tau-1}-1}-\paren{\tau-1}\rho^{\tau-1},\]
and setting it equal to $0$. Solving for $\rho$ gives us $\rho^{\star}=\tau^{1/\paren{\tau-1}}$.

Now, it suffices to show that $\frac{\rho^{\tau}}{\rho^{\tau-1}-1}$ is convex for $\rho>1$ and $\tau\geq 2$. We do this by computing the second derivative, \[\frac{d^2}{d\rho^2}\paren{\frac{\rho^{\tau}}{\rho^{\tau-1}-1}} = \frac{\paren{\tau-1}\rho^\tau\paren{-2\rho^\tau+\tau\paren{\rho+\rho^\tau}}}{\paren{\rho^\tau-\rho}^3},\]
and verifying that it is greater than or equal to $0$. Since $\rho,\tau>1$, we know that \[\frac{\paren{\tau-1}\rho^{\tau+1}}{\paren{\rho^\tau-\rho}^3}\geq 0,\] and so we can remove it from the expression. Now, it suffices to show that
\[-2\rho^{\tau-1}+\tau\paren{1+\rho^{\tau-1}}\geq 0.\]
This inequality can be rewritten as $\tau\geq\paren{2-\tau}\rho^{\tau-1}$ which is trivially true for $\tau\geq 2$.

\subsection{Proof of Lemma \ref{lemma:positive_numerator}}
\label{sec:proof_of_positive_numerator}
Rearranging the inequality gives us
\[b\frac{\rho^b}{\rho^b-1}\geq a\frac{\rho^a}{\rho^a-1}.\]
To show that this holds for $0\leq a\leq b$, it suffices to show that for all $\rho>1$,
\[f\fp{x}:=x\frac{\rho^x}{\rho^x-1}\]
is monotonically increasing for $x>0$ for all $\rho>1$. We show this by checking that $f'\fp{x}\geq 0$. We first rewrite $f\fp{x}$ as follows:
\[f\fp{x}=x\frac{e^{\alpha x}}{e^{\alpha x}-1} = \frac{x}{1-e^{-\alpha x}},\]
where $\alpha = \log \rho > 0$. Then, computing the derivative gives us
\[f'\fp{x}=\frac{\paren{1-e^{-\alpha x}}-\alpha xe^{-\alpha x}}{\paren{1-e^{-\alpha x}}^2}=\frac{e^{-\alpha x}\paren{e^{\alpha x} - 1 - \alpha x}}{\paren{1-e^{-\alpha x}}^2}.\]
Since $e^{\alpha x}\geq 1+\alpha x$ for all $\alpha \geq 0$ and $x\geq 0$, this implies that the numerator, and therefore $f'\fp{x}$, is non-negative.

\section{Pseudocode for Quasiconvex Optimization}
\label{sec:pseudocode_quasiconvex}
We use the notation $\vec{a}_{\alpha}$ to denote a choice of $\vec{a}$ such that Problem \eqref{eq:feasibility} is feasible for this choice of $\alpha$. We then perform binary search to find the smallest value of $\alpha$ such that Problem \eqref{eq:feasibility} is feasible. The algorithm is formalized in Algorithm \ref{alg:binary_search}.

\begin{algorithm}
\KwData{$\epsilon>0, \alpha_{\text{upper}}>1$}
$\alpha_{\text{lower}}\leftarrow 1$\;
 \While{$\alpha_{\text{upper}} - \alpha_{\text{lower}}>\epsilon$}{
 $\alpha \leftarrow \frac{1}{2}\paren{\alpha_{\text{upper}} + \alpha_{\text{lower}}}$\;
  solve Problem \eqref{eq:feasibility} with $\alpha$\;
  \eIf{Problem \eqref{eq:feasibility} is feasible}{
    $\alpha_{\text{upper}}\leftarrow \alpha$\;
    }{
    $\alpha_{\text{lower}}\leftarrow \alpha$\;
    }
  }
  \KwRet{$\vec{a}_{\alpha_{\text{upper}}}$}
 \caption{Quasiconvex Optimization as Convex Feasibility Problems}\label{alg:binary_search}
\end{algorithm}

\section{Missing Proofs in Section \ref{sec:posted_pricing}}
We first introduce the following lemma, which is relevant for the proofs in this section.
% \subsection{Lemma \ref{lemma:v_optimality}}
% \label{sec:v_optimality}
\begin{lemma}
\label{lemma:v_optimality}
Let $v:\R^D_+\rar\R_+$ satisfy Assumption \ref{assump:v} and $\boldsymbol{\lambda}\in\R^D_+$. The optimality conditions of
\[\maximize_{\vec{0}\preceq\vec{x}\preceq\vec{1}} ~v\fp{\vec{x}} - \boldsymbol{\lambda}^\top\vec{x}\]
imply that $\vec{x}^\star=\1{\nabla v\fp{\vec{x}^\star}-\boldsymbol{\lambda}\succeq\vec{0}}$.
\end{lemma}
\begin{proof}
We first re-write the optimization problem by introducing additional constraints, and then apply the KKT conditions.
\begin{align*}
\maximize_{\vec{0}\preceq\vec{x}\preceq\vec{1}}~v\fp{\vec{x}}-\boldsymbol{\lambda}^\top\vec{x}
&= \maximize_{\vec{0}\preceq\vec{x}\preceq\vec{1},\vec{z}}~ v\fp{\vec{z}}-\boldsymbol{\lambda}^\top\vec{x}
\\&\spaces~\text{subject to}~\vec{z}\preceq\vec{x}
\\&= \maximize_{\vec{0}\preceq\vec{x}\preceq\vec{1},\vec{z}}~\minimize_{\vec{w}\succeq\vec{0}}~ v\fp{\vec{z}}-\boldsymbol{\lambda}^\top\vec{x} + \vec{w}^\top\paren{\vec{x}-\vec{z}}
\\&= \maximize_{\vec{0}\preceq\vec{x}\preceq\vec{1},\vec{z}}~\minimize_{\vec{w}\succeq\vec{0}}~\paren{\vec{w}-\boldsymbol{\lambda}}^\top\vec{x} + v\fp{\vec{z}}-\vec{w}^\top\vec{z}.
\end{align*}
The first equality comes from the assumption that $v$ is increasing. Taking the gradient of the Lagrangian with respect to $\vec{z}$ gives the optimality condition: $\nabla v\fp{\vec{z}^\star}-\vec{w}^\star=\vec{0}$. This immediately implies that $\vec{w}^\star=\nabla v\fp{\vec{z}^\star}$. Now, solving the maximization over $\vec{x}$, i.e., $\maximize_{\vec{0}\preceq\vec{x}\preceq\vec{1}}\paren{\vec{w}^\star-\boldsymbol{\lambda}}^\top\vec{x}$, leads to $\vec{x}^\star=\1{\nabla v\fp{\vec{z}^\star}-\boldsymbol{\lambda} \succeq\vec{0}}$. From our constraint on $\vec{z}$, i.e., $\vec{z}\preceq\vec{x}$, we know that $\vec{z}^\star=\vec{x}^\star$ at optimality, and thus plugging this into the threshold rule for $\vec{x}^\star$ gives $\vec{x}^\star=\1{\nabla v\fp{\vec{x}^\star}-\boldsymbol{\lambda} \succeq\vec{0}}$.
\end{proof}

\subsection{Proof of Theorem \ref{thm:competitive_ratio_seq_0}}
\label{sec:proof_of_competitive_ratio_seq_0}
In this subsection, we analyze Algorithm \ref{alg:seq_update_offset}, called with $f_s$ satisfying Assumption \ref{assump:fs}, and $\vec{v}_{\text{offset}}=\vec{0}$. We define the following quantities:
\[P^{s0}:=\sum_{t=1}^Tv_t\fp{\bar{\vec{x}}_t} - f\fp{\sum_{t=1}^T\bar{\vec{x}}_t},\]
\[D^{s0}:=\sum_{t=1}^T\sum_{d=1}^D\max\set{\bracket{\bar{\vec{z}}_t}_d-\bracket{\bar{\boldsymbol{\lambda}}_t}_d,0} - \sum_{t=1}^Tv_{t*}\fp{\bar{\vec{z}}_t} + f^*\fp{\bar{\boldsymbol{\lambda}}_T}.\]
Note that $P^{s0}$ and $D^{s0}$ are identical to their counterparts, $P^s$ and $D^s$ introduced in Section \ref{sec:algorithm}, with the only difference being that for $P^{s0}$ and $D^{s0}$, $\bar{\vec{x}}_t$, $\bar{\boldsymbol{\lambda}}_t$, and $\bar{\vec{z}}_t$ come from Algorithm \ref{alg:seq_update_offset}.
For notational simplicity, we define
\begin{align*}
L_1\fp{g,\set{\vec{x}_t}_{t=1}^T} &:= \sum_{t=1}^T\paren{\nabla g\fp{\sum_{i=1}^t\vec{x}_i}-\nabla g\fp{\sum_{i=1}^{t-1}\vec{x}_i}}^\top\vec{x}_t,
\\L_2\fp{g,\set{\vec{x}_t}_{t=1}^T} &:= \vec{1}^\top\!\paren{\nabla g\fp{\sum_{t=1}^T\vec{x}_t}-\nabla g\fp{\vec{0}}}.
\end{align*}
We see that $L_1\fp{g,\set{\vec{x}_t}_{t=1}^T} \leq L_2\fp{g,\set{\vec{x}_t}_{t=1}^T}$ for all $g$ satisfying Assumption \ref{assump:fs} and for $\vec{0}\preceq\vec{x}_t\preceq\vec{1}$ for all $t\in\bracket{T}$, by upper bounding $\vec{x}_t$ by $\vec{1}$ and performing the telescoping sum, since $g$ is assumed to have an increasing gradient.
We first show the following lemmas which aid in our analysis of the competitive ratio of Algorithm \ref{alg:seq_update_offset}.
\begin{lemma}
\label{lemma:dsdstar_seq_0}
If $f_s$ is increasing and $f_s$ has an increasing gradient, then $D^\star\leq D^{s0}$.
\end{lemma}
\begin{proof}
From the assumption that $f_s$ has an increasing gradient, we know that \[\nabla f_s\fp{\sum_{i=1}^{t-1}\bar{\vec{x}}_i} \preceq \nabla f_s\fp{\sum_{i=1}^{T-1}\bar{\vec{x}}_i}\] for all $t\in\bracket{T}$ since $\bar{\vec{x}}_i\succeq\vec{0}$ for all $i\in\bracket{T}$. This, in turn, implies that $\bar{\boldsymbol{\lambda}}_t \preceq \bar{\boldsymbol{\lambda}}_T$ for all $t\in\bracket{T}$ so that
\begin{align*}
D^{s0} &= \sum_{t=1}^T \sum_{d=1}^D\max\set{\bracket{\bar{\vec{z}}_t}_d-\bracket{\bar{\boldsymbol{\lambda}}_t}_d, 0} - \sum_{t=1}^Tv_{t*}\fp{\bar{\vec{z}}_t} + f^*\fp{\bar{\boldsymbol{\lambda}}_T} \\&
\geq \sum_{t=1}^T \sum_{d=1}^D\max\set{\bracket{\bar{\vec{z}}_t}_d-\bracket{\bar{\boldsymbol{\lambda}}_T}_d, 0} - \sum_{t=1}^Tv_{t*}\fp{\bar{\vec{z}}_t} + f^*\fp{\bar{\boldsymbol{\lambda}}_T} \geq D^\star.
\end{align*}
In order for the final inequality to hold, $\bar{\boldsymbol{\lambda}}_T$ and $\set{\bar{\vec{z}}_t}_{t=1}^T$ must be feasible points for the optimization problem defining $D^\star$, i.e., $\bar{\boldsymbol{\lambda}}_T\succeq\vec{0}$ and $\bar{\vec{z}}_t\succeq\vec{0}$ for all $t\in\bracket{T}$. Since $f_s$ is increasing by assumption, we know that $f_s$ has a non-negative gradient, and similarly since $v_t$ is increasing by assumption, we know that $v_t$ has a non-negative gradient.
% This comes from the assumptions that $f_s$ is increasing and $f_s$ has an increasing gradient.
\end{proof}

\begin{lemma}
\label{lemma:psds_negative_seq_0}
If $f_s$ is increasing and $f_s$ has an increasing gradient, then $P^{s0}\leq D^{s0}$.
\end{lemma}
\begin{proof}
The proof comes from relating $P^{s0}$ to $P^\star$ followed by relating $D^{s0}$ to $D^\star$ with Lemma \ref{lemma:dsdstar_seq_0} and combining the two inequalities using weak duality. Trivially, $P^{s0}\leq P^\star$ since
\[P^{s0}=\sum_{t=1}^Tv_t\fp{\bar{\vec{x}}_t}-f\fp{\sum_{t=1}^T\bar{\vec{x}}_t}\leq \maximize_{\vec{0}\preceq\vec{x}_t\preceq\vec{1}}~ \sum_{t=1}^Tv_t\fp{\vec{x}_t}-f\fp{\sum_{t=1}^T\vec{x}_t} = P^\star.\]
From Lemma \ref{lemma:dsdstar_seq_0}, we have shown that $D^{s0} \geq D^\star$ and with weak duality implying that $P^\star\leq D^\star$, we can conclude that $P^{s0}\leq P^\star\leq D^\star \leq D^{s0}$.
\end{proof}

\begin{lemma}
\label{lemma:objective_positive_seq_0}
Let $v_t$ satisfy Assumption \ref{assump:v}. $f_s$ is convex and differentiable, and $f_s\fp{\vec{0}}=0$, then 
\[\sum_{t=1}^Tv_t\fp{\bar{\vec{x}}_t} - f_s\fp{\sum_{t=1}^T\bar{\vec{x}}_t}+\vec{1}^\top\!\paren{\nabla f_s\fp{\sum_{t=1}^T\bar{\vec{x}}_t}-\nabla f_s\fp{\vec{0}}} \geq 0.\]
\end{lemma}
\begin{proof}
We first note that 
\[L_1\fp{f_s,\set{\bar{\vec{x}}_t}_{t=1}^T} \leq L_2\fp{f_s,\set{\bar{\vec{x}}_t}_{t=1}^T},\]
which allows us to write 
\begin{align*}
&\sum_{t=1}^Tv_t\fp{\bar{\vec{x}}_t} - f_s\fp{\sum_{t=1}^T\bar{\vec{x}}_t}+\vec{1}^\top\!\paren{\nabla f_s\fp{\sum_{t=1}^T\bar{\vec{x}}_t}-\nabla f_s\fp{\vec{0}}}
\\&=\sum_{t=1}^Tv_t\fp{\bar{\vec{x}}_t} - f_s\fp{\sum_{t=1}^T\bar{\vec{x}}_t}+L_2\fp{f_s,\set{\bar{\vec{x}}_t}_{t=1}^T}
\\&\geq\sum_{t=1}^Tv_t\fp{\bar{\vec{x}}_t} - f_s\fp{\sum_{t=1}^T\bar{\vec{x}}_t} + L_1\fp{f_s,\set{\bar{\vec{x}}_t}_{t=1}^T}
% \\&\geq\sum_{t=1}^T\vec{c}_t^\top\bar{\vec{x}}_t - f_s\fp{\sum_{t=1}^T\bar{\vec{x}}_t} + \sum_{t=1}^T\paren{\nabla f_s\fp{\sum_{i=1}^t\bar{\vec{x}}_i}-\nabla f_s\fp{\sum_{i=1}^{t-1}\bar{\vec{x}}_i}}^\top\bar{\vec{x}}_t
\\&\overset{\text{(a)}}{=}\sum_{t=1}^Tv_t\fp{\bar{\vec{x}}_t} - \sum_{t=1}^T\paren{f_s\fp{\sum_{i=1}^t\bar{\vec{x}}_i}-f_s\fp{\sum_{i=1}^{t-1}\bar{\vec{x}}_i}} + L_1\fp{f_s,\set{\bar{\vec{x}}_t}_{t=1}^T}
\\&\overset{\text{(b)}}{\geq} \sum_{t=1}^T\nabla v_t\fp{\bar{\vec{x}}_t}^\top\bar{\vec{x}}_t - \sum_{t=1}^T\nabla f_s\fp{\sum_{i=1}^t\bar{\vec{x}}_i}^\top\bar{\vec{x}}_t + L_1\fp{f_s,\set{\bar{\vec{x}}_t}_{t=1}^T}.
\end{align*}
Equality (a) comes from breaking the expression $f_s\fp{\sum_{t=1}^T\bar{\vec{x}}_t}$ into a telescoping sum with $f_s\fp{\vec{0}}=0$ by assumption. Inequality (b) follows from convexity of $f_s$ and concavity of $v_t$. Now, plugging in the expression for $L_1\fp{f_s,\set{\bar{\vec{x}}_t}_{t=1}^T}$, we get
\begin{align*}
&\sum_{t=1}^Tv_t\fp{\bar{\vec{x}}_t} - f_s\fp{\sum_{t=1}^T\bar{\vec{x}}_t}+\vec{1}^\top\!\paren{\nabla f_s\fp{\sum_{t=1}^T\bar{\vec{x}}_t}-\nabla f_s\fp{\vec{0}}}
\\&\geq \sum_{t=1}^T\nabla v_t\fp{\bar{\vec{x}}_t}^\top\bar{\vec{x}}_t - \sum_{t=1}^T\nabla f_s\fp{\sum_{i=1}^t\bar{\vec{x}}_i}^\top\bar{\vec{x}}_t + L_1\fp{f_s,\set{\bar{\vec{x}}_t}_{t=1}^T}
% \sum_{t=1}^T\vec{c}_t^\top\bar{\vec{x}}_t - \sum_{t=1}^T\nabla f_s\fp{\sum_{i=1}^t\bar{\vec{x}}_i}^\top\bar{\vec{x}}_t + \sum_{t=1}^T\paren{\nabla f_s\fp{\sum_{i=1}^t\bar{\vec{x}}_i}-\nabla f_s\fp{\sum_{i=1}^{t-1}\bar{\vec{x}}_i}}^\top\bar{\vec{x}}_t
\\&= \sum_{t=1}^T\nabla v_t\fp{\bar{\vec{x}}_t}^\top\bar{\vec{x}}_t-\sum_{t=1}^T\nabla f_s\fp{\sum_{i=1}^{t-1}\bar{\vec{x}}_i}^\top\bar{\vec{x}}_t
\\&=\sum_{t=1}^T\inprod{\nabla v_t\fp{\bar{\vec{x}}_t}-\nabla f_s\fp{\sum_{i=1}^{t-1}\bar{\vec{x}}_i}}{\bar{\vec{x}}_t}.
\end{align*}
The decision rule of Algorithm \ref{alg:seq_update_offset}---i.e., $\bar{\vec{x}}_t =\1{\bar{\vec{z}}_t - \bar{\boldsymbol{\lambda}}_t \succeq \vec{0}}$ (as seen in Lemma \ref{lemma:v_optimality}) ensures that the inner product is always non-negative.
\end{proof}

Now, we bound the competitive ratio of Algorithm \ref{alg:seq_update_offset}. The general overview of the proof of Theorem \ref{thm:competitive_ratio_seq_0} is as follows: writing $D^{s0}$ in terms of $P^{s0}$, we bound the gap between $D^{s0}$ and $P^{s0}$. From here, we lower bound $D^{s0}$ by $D^{\star}$, which in turn allows us to use weak duality to relate $D^{\star}$ and $P^{\star}$.

We start with writing $D^{s0}$ in terms of $P^{s0}$:
\begin{align*}
D^{s0} &= \sum_{t=1}^T \sum_{d=1}^D\max\set{\bracket{\bar{\vec{z}}_t}_d-\bracket{\bar{\boldsymbol{\lambda}}_t}_d, 0} - \sum_{t=1}^Tv_{t*}\fp{\bar{\vec{z}}_t} + f^*\fp{\bar{\boldsymbol{\lambda}}_T}
\\&\overset{\text{(a)}}{=} \sum_{t=1}^T\paren{\bar{\vec{z}}_t-\bar{\boldsymbol{\lambda}}_t}^\top\bar{\vec{x}}_t - \sum_{t=1}^Tv_{t*}\fp{\bar{\vec{z}}_t} + f^*\fp{\bar{\boldsymbol{\lambda}}_T}
\\&\overset{\text{(b)}}{=} \sum_{t=1}^T\nabla v_t\fp{\bar{\vec{x}}_t}^\top\bar{\vec{x}}_t - \sum_{t=1}^T\nabla f_s\fp{\sum_{i=1}^{t-1}\bar{\vec{x}}_i}^\top\bar{\vec{x}}_t - \sum_{t=1}^Tv_{t*}\fp{\nabla v_t\fp{\bar{\vec{x}}_t}} + f^*\fp{\bar{\boldsymbol{\lambda}}_T}
\\&\overset{\text{(c)}}{=} \sum_{t=1}^T\nabla v_t\fp{\bar{\vec{x}}_t}^\top\bar{\vec{x}}_t - \sum_{t=1}^T\nabla f_s\fp{\sum_{i=1}^t\bar{\vec{x}}_i}^\top\bar{\vec{x}}_t - \sum_{t=1}^Tv_{t*}\fp{\nabla v_t\fp{\bar{\vec{x}}_t}} + f^*\fp{\bar{\boldsymbol{\lambda}}_T}+L_1\fp{f_s, \set{\bar{\vec{x}}_t}_{t=1}^T}.
\end{align*}
Equality (a) comes from the decision rule of Algorithm \ref{alg:seq_update_offset}, which ensures that $\bar{\vec{x}}_t = \1{\bar{\vec{z}}_t - \bar{\boldsymbol{\lambda}}_t \succeq \vec{0}}$. Equality (b) comes from first substituting the definition of $\bar{\boldsymbol{\lambda}}_t = \nabla f_s\fp{\sum_{i=1}^{t-1}\bar{\vec{x}}_i}$ and $\bar{\vec{z}}_t=\nabla v_t\fp{\bar{\vec{x}}_t}$, and equality (c) comes from adding and subtracting $\sum_{t=1}^T\nabla f_s\fp{\sum_{i=1}^t\bar{\vec{x}}_i}^\top\bar{\vec{x}}_t$.
% Equality (c) comes from the increasing gradient property of $f_s$.
Now, we apply the concave Fenchel-Young inequality at equality---i.e., Equation \eqref{eq:fenchel-young-concave} with $g=v_t$ and $\vec{u}=\bar{\vec{x}}_t$, in order to decompose the $v_{t*}\fp{\nabla v_t\fp{\bar{\vec{x}}_t}}$ term as follows:
\begin{align*}
D^{s0} &= \sum_{t=1}^T\nabla v_t\fp{\bar{\vec{x}}_t}^\top\bar{\vec{x}}_t - \sum_{t=1}^T\nabla f_s\fp{\sum_{i=1}^t\bar{\vec{x}}_i}^\top\bar{\vec{x}}_t - \sum_{t=1}^T\paren{\nabla v_t\fp{\bar{\vec{x}}_t}^\top\bar{\vec{x}}_t-v_t\fp{\bar{\vec{x}}_t}} + f^*\fp{\bar{\boldsymbol{\lambda}}_T}+L_1\fp{f_s, \set{\bar{\vec{x}}_t}_{t=1}^T}
\\&= \sum_{t=1}^Tv_t\fp{\bar{\vec{x}}_t} - \sum_{t=1}^T\nabla f_s\fp{\sum_{i=1}^t\bar{\vec{x}}_i}^\top\bar{\vec{x}}_t + f^*\fp{\bar{\boldsymbol{\lambda}}_T}+L_1\fp{f_s, \set{\bar{\vec{x}}_t}_{t=1}^T}.
\end{align*}
Now, we proceed to bound the duality gap between $D^{s0}$ and $P^{s0}$ by first observing the following relationship:
\begin{align*}
D^{s0} &\overset{\text{(d)}}{\leq} \sum_{t=1}^Tv_t\fp{\bar{\vec{x}}_t} - f_s\fp{\sum_{t=1}^T\bar{\vec{x}}_t} + f^*\fp{\bar{\boldsymbol{\lambda}}_T} + L_1\fp{f_s, \set{\bar{\vec{x}}_t}_{t=1}^T}
\\&= \sum_{t=1}^Tv_t\fp{\bar{\vec{x}}_t} - f_s\fp{\sum_{t=1}^T\bar{\vec{x}}_t} + f^*\fp{\bar{\boldsymbol{\lambda}}_T} + f\fp{\sum_{t=1}^T\bar{\vec{x}}_t} - f\fp{\sum_{t=1}^T\bar{\vec{x}}_t}+L_1\fp{f_s, \set{\bar{\vec{x}}_t}_{t=1}^T}
\\&\overset{\text{(e)}}{=} P^{s0} - f_s\fp{\sum_{t=1}^T\bar{\vec{x}}_t} + f^*\fp{\bar{\boldsymbol{\lambda}}_T} + f\fp{\sum_{t=1}^T\bar{\vec{x}}_t} + L_1\fp{f_s, \set{\bar{\vec{x}}_t}_{t=1}^T}
% \\&\leq P^{s0} - \max\set{0,f_s\fp{\sum_{t=1}^T\bar{\vec{x}}_t} - f^*\fp{\bar{\boldsymbol{\lambda}}_T} - f\fp{\sum_{t=1}^T\bar{\vec{x}}_t}}.
% \\&\overset{\text{(f)}}{\leq}P^{s0} - f_s\fp{\sum_{t=1}^T\bar{\vec{x}}_t} + f^*\fp{\bar{\boldsymbol{\lambda}}_{T+1}} + f\fp{\sum_{t=1}^T\bar{\vec{x}}_t} + L_1\fp{f_s, \set{\bar{\vec{x}}_t}_{t=1}^T}
\\&\overset{\text{(f)}}{\leq}P^{s0} - f_s\fp{\sum_{t=1}^T\bar{\vec{x}}_t} + f^*\fp{\bar{\boldsymbol{\lambda}}_T} + f\fp{\sum_{t=1}^T\bar{\vec{x}}_t} + L_2\fp{f_s, \set{\bar{\vec{x}}_t}_{t=1}^T}.
\end{align*}
Inequality (d) follows directly from convexity of $f_s$ and concavity of $v_t$. Equality (e) follows from substituting the definition of $P^{s0} = \sum_{t=1}^Tv_t\fp{\bar{\vec{x}}_t} - f\fp{\sum_{t=1}^T\bar{\vec{x}}_t}$, where in the preceding equality we add and subtract $f\fp{\sum_{t=1}^T\bar{\vec{x}}_t}$. Inequality (f) follows from \[L_1\fp{f_s, \set{\bar{\vec{x}}_t}_{t=1}^T} \leq L_2\fp{f_s, \set{\bar{\vec{x}}_t}_{t=1}^T}.\] We bound the gap between $D^{s0}$ and $P^{s0}$ as a multiplicative factor of $P^{s0}$ in order to relate these quantities as a ratio:
\begin{align*}
&\frac{f_s\fp{\sum_{t=1}^T\bar{\vec{x}}_t} - f^*\fp{\bar{\boldsymbol{\lambda}}_T} - f\fp{\sum_{t=1}^T\bar{\vec{x}}_t} - L_2\fp{f_s, \set{\bar{\vec{x}}_t}_{t=1}^T}}{P^{s0}}
\\&\overset{\text{(g)}}{\geq}\frac{f_s\fp{\sum_{t=1}^T\bar{\vec{x}}_t} - f^*\fp{\bar{\boldsymbol{\lambda}}_{T+1}} - f\fp{\sum_{t=1}^T\bar{\vec{x}}_t} - L_2\fp{f_s, \set{\bar{\vec{x}}_t}_{t=1}^T}}{P^{s0}}
\\&\overset{\text{(h)}}{=} \frac{f_s\fp{\sum_{t=1}^T\bar{\vec{x}}_t} - f^*\fp{\nabla f_s\fp{\sum_{t=1}^T\bar{\vec{x}}_t}} - f\fp{\sum_{t=1}^T\bar{\vec{x}}_t} - L_2\fp{f_s, \set{\bar{\vec{x}}_t}_{t=1}^T}}{\sum_{t=1}^Tv_t\fp{\bar{\vec{x}}_t} - f\fp{\sum_{t=1}^T\bar{\vec{x}}_t}}
\\&\overset{\text{(i)}}{\geq} \frac{f_s\fp{\sum_{t=1}^T\bar{\vec{x}}_t} - f^*\fp{\nabla f_s\fp{\sum_{t=1}^T\bar{\vec{x}}_t}} - f\fp{\sum_{t=1}^T\bar{\vec{x}}_t} - L_2\fp{f_s, \set{\bar{\vec{x}}_t}_{t=1}^T}}{f_s\fp{\sum_{t=1}^T\bar{\vec{x}}_t} - f\fp{\sum_{t=1}^T\bar{\vec{x}}_t} - L_2\fp{f_s, \set{\bar{\vec{x}}_t}_{t=1}^T}}
% \\&= \frac{f_s\fp{\sum_{t=1}^{T-1}\bar{\vec{x}}_t+\bar{\vec{x}}_T} - f^*\fp{\nabla f_s\fp{\sum_{t=1}^{T-1}\bar{\vec{x}}_t}} - f\fp{\sum_{t=1}^{T-1}\bar{\vec{x}}_t+\bar{\vec{x}}_T}+\sum_{t=1}^T\paren{\bar{\boldsymbol{\lambda}}_t-\bar{\boldsymbol{\lambda}}^{\paren{t+1}}}^\top\bar{\vec{x}}_t}{f_s\fp{\sum_{t=1}^{T-1}\bar{\vec{x}}_t+\bar{\vec{x}}_T}+\sum_{t=1}^T\paren{\bar{\boldsymbol{\lambda}}_t-\bar{\boldsymbol{\lambda}}^{\paren{t+1}}}^\top\bar{\vec{x}}_t - f\fp{\sum_{t=1}^{T-1}\bar{\vec{x}}_t+\bar{\vec{x}}_T}}
\\&\overset{\text{(j)}}{\geq}\inf_{\vec{0}\preceq\vec{u}\preceq T\vec{1}} \frac{f_s\fp{\vec{u}} - f^*\fp{\nabla f_s\fp{\vec{u}}} - f\fp{\vec{u}} - \vec{1}^\top\!\paren{\nabla f_s\fp{\vec{u}}-\nabla f_s\fp{\vec{0}}}}{f_s\fp{\vec{u}} - f\fp{\vec{u}} - \vec{1}^\top\!\paren{\nabla f_s\fp{\vec{u}}-\nabla f_s\fp{\vec{0}}}}
% \\&\overset{\text{(i)}}{\geq}\inf_{\vec{0}\preceq\vec{u}\preceq \paren{T-1}\vec{1}} \frac{f_s\fp{\vec{u}} - f^*\fp{\nabla f_s\fp{\vec{u}}} - f\fp{\vec{u}}+\sum_{t=1}^T\paren{\bar{\boldsymbol{\lambda}}_t-\bar{\boldsymbol{\lambda}}^{\paren{t+1}}}^\top\bar{\vec{x}}_t}{f_s\fp{\vec{u}}+\sum_{t=1}^T\paren{\bar{\boldsymbol{\lambda}}_t-\bar{\boldsymbol{\lambda}}^{\paren{t+1}}}^\top\bar{\vec{x}}_t - f\fp{\vec{u}}}
=:\beta_{f,f_s}.
\end{align*}
Inequality (g) follows from $f^*$ being an increasing function and the assumption that $f_s$ has an increasing gradient implying that $\bar{\boldsymbol{\lambda}}_T \preceq \bar{\boldsymbol{\lambda}}_{T+1}$. In equality (h), we substitute the definition of $\bar{\boldsymbol{\lambda}}_{T+1} = \nabla f_s\fp{\sum_{i=1}^T\bar{\vec{x}}_i}$. Inequality (i) follows from replacing $\sum_{t=1}^Tv_t\fp{\bar{\vec{x}}_t}$ with $f_s\fp{\sum_{t=1}^T\bar{\vec{x}}_t}-L_2\fp{f_s, \set{\bar{\vec{x}}_t}_{t=1}^T}$ in the denominator. This creates a lower bound because Lemma \ref{lemma:psds_negative_seq_0} shows that the numerator is non-positive and Lemma \ref{lemma:objective_positive_seq_0} shows that $f_s\fp{\sum_{t=1}^T\bar{\vec{x}}_t}-L_2\fp{f_s, \set{\bar{\vec{x}}_t}_{t=1}^T} \leq \sum_{t=1}^Tv_t\fp{\bar{\vec{x}}_t}$.
Inequality (j) follows from observing that $\vec{0}\preceq \sum_{t=1}^T\bar{\vec{x}}_t\preceq T\vec{1}$. 
% Inequality (i) comes from seeing that $\frac{b-a}{b}$ is increasing in $b$ (here $f_s\fp{\vec{u}}-f\fp{\vec{u}} = b$ and $f^*\fp{\nabla f_s\fp{\vec{u}}}+\sum_{t=1}^T\paren{\nabla f_s\fp{\sum_{i=1}^{t-1}\bar{\vec{x}}_i}-\nabla f_s\fp{\sum_{i=1}^t\bar{\vec{x}}_i}}=a$). 
Hence,
\[\beta_{f,f_s}P^{s0}\leq f_s\fp{\sum_{t=1}^T\bar{\vec{x}}_t} - f^*\fp{\bar{\boldsymbol{\lambda}}_T} - f\fp{\sum_{t=1}^T\bar{\vec{x}}_t}+L_2\fp{f_s, \set{\bar{\vec{x}}_t}_{t=1}^T}.\]
Define 
\begin{align*}
\alpha_{f,f_s} :=1-\beta_{f,f_s} = \sup_{\vec{0}\preceq\vec{u}\preceq T\vec{1}} \frac{f^*\fp{\nabla f_s\fp{\vec{u}}}}{f_s\fp{\vec{u}} - f\fp{\vec{u}} - \vec{1}^\top\!\paren{\nabla f_s\fp{\vec{u}}-\nabla f_s\fp{\vec{0}}}}.
\end{align*}
% Note that the assumption that $f_s\fp{\vec{u}}\geq f\fp{\vec{u}}$ for all $\vec{0}\preceq\vec{u}\preceq T\vec{1}$ ensures that $\alpha_{f,f_s}\geq 0$, which ensures that the competitive ratio is non-negative. 
We lower bound $D^{s0}$ by $D^{\star}$ by Lemma \ref{lemma:dsdstar_seq_0} so $P^{s0} - D^{s0} \geq P^{s0}\beta_{f,f_s}$. Applying weak duality, we get that $P^{s0} - P^\star \geq P^{s0}\beta_{f,f_s}$. Rearranging this equation gives us the following:
\[\frac{P^{s0}}{P^\star} \geq \frac{1}{1-\beta_{f,f_s}} = \frac{1}{\alpha_{f,f_s}}.\]
This concludes the proof of Theorem \ref{thm:competitive_ratio_seq_0}.

\subsection{Proof of Theorem \ref{thm:quasiconvex_seq_0}}
\label{sec:proof_of_quasiconvex_seq_0}
In order to show that Problem \eqref{eq:quasiconvex_seq_0} is a quasiconvex optimization problem, we must verify that the constraints are convex and the objective is quasiconvex.
It suffices to show that
\[\max_{\vec{u}\in\U}~\frac{f^*\fp{\nabla f_s\fp{\vec{u}}}}{f_s\fp{\vec{u}}-f\fp{\vec{u}}-\vec{1}^\top\!\paren{\nabla f_s\fp{\vec{u}}-\nabla f_s\fp{\vec{0}}}}\]
is a quasiconvex function in $\vec{a}$. Since a non-negative weighted maximum of quasiconvex functions is also quasiconvex, it suffices to show that $\frac{f^*\fp{\nabla f_s\fp{\vec{u}}}}{f_s\fp{\vec{u}}-f\fp{\vec{u}}-\vec{1}^\top\!\paren{\nabla f_s\fp{\vec{u}}-\nabla f_s\fp{\vec{0}}}}$ is quasiconvex in $\vec{a}$ for a fixed $\vec{u}$. We can directly apply the definition of quasiconvexity. Let $S_{\alpha}\fp{f_s}$ be the sub-level sets of $f_s$ for $\vec{a}\in\R^N$. We have the following:
\begin{align*}
S_{\alpha}\fp{f_s}&=\setcond{\vec{a}\succeq\vec{1}}{\frac{f^*\fp{\nabla f_s\fp{\vec{u}}}}{f_s\fp{\vec{u}}-f\fp{\vec{u}} - \vec{1}^\top\!\paren{\nabla f_s\fp{\vec{u}}-\nabla f_s\fp{\vec{0}}}}\leq\alpha}
\\&= \setcond{\vec{a}\succeq\vec{1}}{f^*\fp{\nabla f_s\fp{\vec{u}}}\leq\alpha\paren{f_s\fp{\vec{u}}-f\fp{\vec{u}} - \vec{1}^\top\!\paren{\nabla f_s\fp{\vec{u}}-\nabla f_s\fp{\vec{0}}}}}.
\end{align*}
For a fixed value of $\vec{u}$, $\bracket{\nabla f_s\fp{\vec{u}}}_d$ is linear in $\vec{a}$ for all $d$. Since $f^*$ is always convex, composing a convex function with a linear function of $\vec{a}$ is convex in $\vec{a}$. Finally, since $f_s\fp{\vec{u}}$ is linear in $\vec{a}$, 
\[f_s\fp{\vec{u}}-f\fp{\vec{u}} - \vec{1}^\top\!\paren{\nabla f_s\fp{\vec{u}}-\nabla f_s\fp{\vec{0}}}\]
must be linear in $\vec{a}$, and thus the constraints of $S_{\alpha}\fp{f_s}$ are convex, directly implying that $S_{\alpha}\fp{f_s}$ is a convex set.

\subsection{Proof of Theorem \ref{thm:competitive_ratio_seq_1}}
\label{sec:proof_of_competitive_ratio_seq_1}
In this subsection, we analyze Algorithm \ref{alg:seq_update_offset} called with $f_s$ satisfying Assumption \ref{assump:fs_seq} and $\vec{v}_{\text{offset}}=\vec{1}$.
We define the following quantities:
\[P^{s1}:=\sum_{t=1}^Tv_t\fp{\bar{\vec{x}}_t}-f\fp{\sum_{t=1}^T\bar{\vec{x}}_t},\]
\[D^{s1}:=\sum_{t=1}^T\sum_{d=1}^D\max\set{\bracket{\bar{\vec{z}}_t}_d-\bracket{\bar{\boldsymbol{\lambda}}_t}_d,0}-\sum_{t=1}^Tv_{t*}\fp{\bar{\vec{z}}_t}+f^*\fp{\bar{\boldsymbol{\lambda}}_T}.\]
Note that $P^{s1}$ and $D^{s1}$ are identical to their counterparts, $P^s$ and $D^s$ introduced in Section \ref{sec:algorithm}, with the only difference being that for $P^{s1}$ and $D^{s1}$, $\bar{\vec{x}}_t$, $\bar{\boldsymbol{\lambda}}_t$, and $\bar{\vec{z}}_t$ come from Algorithm \ref{alg:seq_update_offset}.

We first show the following lemmas which aid in our analysis of the competitive ratio of Algorithm \ref{alg:seq_update_offset}.
\begin{lemma}
\label{lemma:dsdstar_seq_1}
If $f_s$ is increasing and $f_s$ has an increasing gradient, then $D^\star\leq D^{s1}$.
\end{lemma}
\begin{proof}
From the assumption that $f_s$ has an increasing gradient, we know that \[\nabla f_s\fp{\sum_{i=1}^{t-1}\bar{\vec{x}}_i+ \vec{1}} \preceq \nabla f_s\fp{\sum_{i=1}^{T-1}\bar{\vec{x}}_i+ \vec{1}}\] for all $t\in\bracket{T}$ since $\bar{\vec{x}}_i\succeq\vec{0}$ for all $i\in\bracket{T}$. This, in turn, implies that $\bar{\boldsymbol{\lambda}}_t \preceq \bar{\boldsymbol{\lambda}}_T$ for all $t\in\bracket{T}$ so that
\begin{align*}
D^{s1} &= \sum_{t=1}^T \sum_{d=1}^D\max\set{\bracket{\bar{\vec{z}}_t}_d-\bracket{\bar{\boldsymbol{\lambda}}_t}_d, 0} - \sum_{t=1}^Tv_{t*}\fp{\bar{\vec{z}}_t} + f^*\fp{\bar{\boldsymbol{\lambda}}_T} \\&
\geq \sum_{t=1}^T \sum_{d=1}^D\max\set{\bracket{\bar{\vec{z}}_t}_d-\bracket{\bar{\boldsymbol{\lambda}}_T}_d, 0} - \sum_{t=1}^Tv_{t*}\fp{\bar{\vec{z}}_t} + f^*\fp{\bar{\boldsymbol{\lambda}}_T} \geq D^\star.
\end{align*}
% In order for the final inequality to hold, $\bar{\boldsymbol{\lambda}}_T$ must be a feasible point for the optimization problem defining $D^\star$, i.e., $\bar{\boldsymbol{\lambda}}_T\succeq\vec{0}$. This comes from the assumptions that $f_s$ is increasing and $f_s$ has an increasing gradient.
In order for the final inequality to hold, $\bar{\boldsymbol{\lambda}}_T$ and $\set{\bar{\vec{z}}_t}_{t=1}^T$ must be feasible points for the optimization problem defining $D^\star$, i.e., $\bar{\boldsymbol{\lambda}}_T\succeq\vec{0}$ and $\bar{\vec{z}}_t\succeq\vec{0}$ for all $t\in\bracket{T}$. Since $f_s$ is increasing by assumption, we know that $f_s$ has a non-negative gradient, and similarly since $v_t$ is increasing by assumption, we know that $v_t$ has a non-negative gradient.
\end{proof}

\begin{lemma}
\label{lemma:psds_negative_seq_1}
If $f_s$ is increasing and $f_s$ has an increasing gradient, then $P^{s1}\leq D^{s1}$.
\end{lemma}
\begin{proof}
The proof comes from relating $P^{s1}$ to $P^\star$ followed by relating $D^{s1}$ to $D^\star$ with Lemma \ref{lemma:dsdstar_seq_1} and combining the two inequalities using weak duality. Trivially, $P^{s1}\leq P^\star$ since
\[P^{s1}=\sum_{t=1}^Tv_t\fp{\bar{\vec{x}}_t}-f\fp{\sum_{t=1}^T\bar{\vec{x}}_t}\leq \maximize_{\vec{0}\preceq\vec{x}_t\preceq\vec{1}}~ \sum_{t=1}^Tv_t\fp{\vec{x}_t}-f\fp{\sum_{t=1}^T\vec{x}_t} = P^\star.\]
From Lemma \ref{lemma:dsdstar_seq_1}, we have shown that $D^{s1} \geq D^\star$ and with weak duality implying that $P^\star\leq D^\star$, we can conclude that $P^{s1}\leq P^\star\leq D^\star \leq D^{s1}$.
\end{proof}

We now show that Algorithm \ref{alg:seq_update_offset} does not make a decision which causes the objective to become negative.
\begin{lemma}
\label{lemma:objective_positive_seq_1}
Let $v_t$ satisfy Assumption \ref{assump:v}. If $f_s$ is convex and differentiable, $f_s$ has an increasing gradient, and $f_s\fp{\vec{0}}=0$, then 
\[\sum_{t=1}^Tv_t\fp{\bar{\vec{x}}_t} - f_s\fp{\sum_{t=1}^T\bar{\vec{x}}_t}\geq 0.\]
\end{lemma}
\begin{proof}
Since $f_s\fp{\vec{0}}=0$, we have that 
\begin{align*}
\sum_{t=1}^Tv_t\fp{\bar{\vec{x}}_t} - f_s\fp{\sum_{t=1}^T\bar{\vec{x}}_t}
&\overset{\text{(a)}}{\geq} \sum_{t=1}^T\nabla v_t\fp{\bar{\vec{x}}_t}^\top\bar{\vec{x}}_t - f_s\fp{\sum_{t=1}^T\bar{\vec{x}}_t}
\\&=\sum_{t=1}^T\paren{\nabla v_t\fp{\bar{\vec{x}}_t}^\top\bar{\vec{x}}_t - f_s\fp{\sum_{i=1}^t\bar{\vec{x}}_i} + f_s\fp{\sum_{i=1}^{t-1}\bar{\vec{x}}_i}}
\\&\overset{\text{(b)}}{\geq} \sum_{t=1}^T\inprod{\nabla v_t\fp{\bar{\vec{x}}_t} - \nabla f_s\fp{\sum_{i=1}^t\bar{\vec{x}}_i}}{\bar{\vec{x}}_t}
\\&\overset{\text{(c)}}{\geq}\sum_{t=1}^T\inprod{\nabla v_t\fp{\bar{\vec{x}}_t} - \nabla f_s\fp{\sum_{i=1}^{t-1}\bar{\vec{x}}_i+\vec{1}}}{\bar{\vec{x}}_t}.
\end{align*}
Inequality (a) comes from concavity of $v_t$ and inequality (b) follows from convexity of $f_s$. Inequality (c) comes from $f_s$ having an increasing gradient. 
Finally, the decision rule of Algorithm \ref{alg:seq_update_offset}---i.e., $\bar{\vec{x}}_t =\1{\bar{\vec{z}}_t - \bar{\boldsymbol{\lambda}}_t \succeq \vec{0}}$ (as seen in Lemma \ref{lemma:v_optimality}) ensures that the inner product is always non-negative.
\end{proof}

Now, we bound the competitive ratio of Algorithm \ref{alg:seq_update_offset}. The general overview of the proof of Theorem \ref{thm:competitive_ratio_seq_1} is as follows: writing $D^{s1}$ in terms of $P^{s1}$, we bound the gap between $D^{s1}$ and $P^{s1}$. From here, we lower bound $D^{s1}$ by $D^{\star}$, which in turn allows us to use weak duality to relate $D^{\star}$ and $P^{\star}$.

We start with writing $D^{s1}$ in terms of $P^{s1}$:
\begin{align*}
D^{s1} &= \sum_{t=1}^T \sum_{d=1}^D\max\set{\bracket{\bar{\vec{z}}_t}_d-\bracket{\bar{\boldsymbol{\lambda}}_t}_d, 0} - \sum_{t=1}^Tv_{t*}\fp{\bar{\vec{z}}_t} + f^*\fp{\bar{\boldsymbol{\lambda}}_T}
\\&\overset{\text{(a)}}{=} \sum_{t=1}^T\paren{\bar{\vec{z}}_t-\bar{\boldsymbol{\lambda}}_t}^\top\bar{\vec{x}}_t - \sum_{t=1}^Tv_{t*}\fp{\bar{\vec{z}}_t} + f^*\fp{\bar{\boldsymbol{\lambda}}_T}
% = \sum_{t=1}^T\paren{\vec{c}_t-\nabla f_s\fp{\sum_{i=1}^t\bar{\vec{x}}_i}}^\top\bar{\vec{x}}_t + f^*\fp{\bar{\boldsymbol{\lambda}}_T}
\\&\overset{\text{(b)}}{=} \sum_{t=1}^T\nabla v_t\fp{\bar{\vec{x}}_t}^\top\bar{\vec{x}}_t - \sum_{t=1}^T\nabla f_s\fp{\sum_{i=1}^{t-1}\bar{\vec{x}}_i+\vec{1}}^\top\bar{\vec{x}}_t - \sum_{t=1}^Tv_{t*}\fp{\nabla v_t\fp{\bar{\vec{x}}_t}} + f^*\fp{\bar{\boldsymbol{\lambda}}_T}
\\&\overset{\text{(c)}}{\leq}\sum_{t=1}^T\nabla v_t\fp{\bar{\vec{x}}_t}^\top\bar{\vec{x}}_t - \sum_{t=1}^T\nabla f_s\fp{\sum_{i=1}^t\bar{\vec{x}}_i}^\top\bar{\vec{x}}_t - \sum_{t=1}^Tv_{t*}\fp{\nabla v_t\fp{\bar{\vec{x}}_t}} + f^*\fp{\bar{\boldsymbol{\lambda}}_T}.
\end{align*}
Equality (a) comes from the decision rule of Algorithm \ref{alg:seq_update_offset}, which ensures that $\bar{\vec{x}}_t = \1{\bar{\vec{z}}_t - \bar{\boldsymbol{\lambda}}_t \succeq \vec{0}}$. Equality (b) comes from substituting the definition of $\bar{\boldsymbol{\lambda}}_t = \nabla f_s\fp{\sum_{i=1}^{t-1}\bar{\vec{x}}_i+\vec{1}}$.
% replacing $\bar{\boldsymbol{\lambda}}_t$ with $\nabla f_s\fp{\sum_{i=1}^{t}\bar{\vec{x}}_i+\vec{1}}$.
Inequality (c) comes from the increasing gradient property of $f_s$.
Now, we apply the concave Fenchel-Young inequality at equality---i.e., Equation \eqref{eq:fenchel-young-concave} with $g=v_t$ and $\vec{u}=\bar{\vec{x}}_t$, in order to decompose the $v_{t*}\fp{\nabla v_t\fp{\bar{\vec{x}}_t}}$ term as follows:
\begin{align*}
D^{s1}&\leq \sum_{t=1}^T\nabla v_t\fp{\bar{\vec{x}}_t}^\top\bar{\vec{x}}_t - \sum_{t=1}^T\nabla f_s\fp{\sum_{i=1}^t\bar{\vec{x}}_i}^\top\bar{\vec{x}}_t - \sum_{t=1}^Tv_{t*}\fp{\nabla v_t\fp{\bar{\vec{x}}_t}} + f^*\fp{\bar{\boldsymbol{\lambda}}_T}
\\&= \sum_{t=1}^T\nabla v_t\fp{\bar{\vec{x}}_t}^\top\bar{\vec{x}}_t - \sum_{t=1}^T\nabla f_s\fp{\sum_{i=1}^t\bar{\vec{x}}_i}^\top\bar{\vec{x}}_t - \sum_{t=1}^T\paren{\nabla v_t\fp{\bar{\vec{x}}_t}^\top\bar{\vec{x}}_t-v_t\fp{\bar{\vec{x}}_t}} + f^*\fp{\bar{\boldsymbol{\lambda}}_T}
\\&= \sum_{t=1}^Tv_t\fp{\bar{\vec{x}}_t} - \sum_{t=1}^T\nabla f_s\fp{\sum_{i=1}^t\bar{\vec{x}}_i}^\top\bar{\vec{x}}_t + f^*\fp{\bar{\boldsymbol{\lambda}}_T}
\end{align*}
Now, we proceed to bound the duality gap between $D^{s1}$ and $P^{s1}$ by first observing the following relationship:
\begin{align*}
D^{s1} &\overset{\text{(d)}}{\leq} \sum_{t=1}^Tv_t\fp{\bar{\vec{x}}_t} - f_s\fp{\sum_{t=1}^T\bar{\vec{x}}_t} + f^*\fp{\bar{\boldsymbol{\lambda}}_T}
\\&= \sum_{t=1}^Tv_t\fp{\bar{\vec{x}}_t} - f_s\fp{\sum_{t=1}^T\bar{\vec{x}}_t} + f^*\fp{\bar{\boldsymbol{\lambda}}_T} + f\fp{\sum_{t=1}^T\bar{\vec{x}}_t} - f\fp{\sum_{t=1}^T\bar{\vec{x}}_t}
\\&\overset{\text{(e)}}{=} P^{s1} - f_s\fp{\sum_{t=1}^T\bar{\vec{x}}_t} + f^*\fp{\bar{\boldsymbol{\lambda}}_T} + f\fp{\sum_{t=1}^T\bar{\vec{x}}_t}.
% \\&\leq P^{s1} - \max\set{0,f_s\fp{\sum_{t=1}^T\bar{\vec{x}}_t} - f^*\fp{\bar{\boldsymbol{\lambda}}_T} - f\fp{\sum_{t=1}^T\bar{\vec{x}}_t}}.
\end{align*}
Inequality (d) follows directly from convexity of $f_s$ and (e) follows by substituting the definition of $P^{s1} = \sum_{t=1}^Tv_t\fp{\bar{\vec{x}}_t} - f\fp{\sum_{t=1}^T\bar{\vec{x}}_t}$ where in the preceding equality we add and subtract $f\fp{\sum_{t=1}^T\bar{\vec{x}}_t}$. We bound the gap between $D^{s1}$ and $P^{s1}$ as a multiplicative factor of $P^{s1}$ in order to relate these quantities as a ratio:
\begin{align*}
&\frac{f_s\fp{\sum_{t=1}^T\bar{\vec{x}}_t} - f^*\fp{\bar{\boldsymbol{\lambda}}_T} - f\fp{\sum_{t=1}^T\bar{\vec{x}}_t}}{P^{s1}}
\\&\overset{\text{(f)}}{=} \frac{f_s\fp{\sum_{t=1}^T\bar{\vec{x}}_t} - f^*\fp{\nabla f_s\fp{\sum_{t=1}^{T-1}\bar{\vec{x}}_t+\vec{1}}} - f\fp{\sum_{t=1}^T\bar{\vec{x}}_t}}{\sum_{t=1}^Tv_t\fp{\bar{\vec{x}}_t} - f\fp{\sum_{t=1}^T\bar{\vec{x}}_t}}
\\&\overset{\text{(g)}}{\geq} \frac{f_s\fp{\sum_{t=1}^T\bar{\vec{x}}_t} - f^*\fp{\nabla f_s\fp{\sum_{t=1}^{T-1}\bar{\vec{x}}_t+\vec{1}}} - f\fp{\sum_{t=1}^T\bar{\vec{x}}_t}}{f_s\fp{\sum_{t=1}^T\bar{\vec{x}}_t} - f\fp{\sum_{t=1}^T\bar{\vec{x}}_t}}
\\&= \frac{f_s\fp{\sum_{t=1}^{T-1}\bar{\vec{x}}_t+\bar{\vec{x}}_T} - f^*\fp{\nabla f_s\fp{\sum_{t=1}^{T-1}\bar{\vec{x}}_t+\vec{1}}} - f\fp{\sum_{t=1}^{T-1}\bar{\vec{x}}_t+\bar{\vec{x}}_T}}{f_s\fp{\sum_{t=1}^{T-1}\bar{\vec{x}}_t+\bar{\vec{x}}_T} - f\fp{\sum_{t=1}^{T-1}\bar{\vec{x}}_t+\bar{\vec{x}}_T}}
\\&\overset{\text{(h)}}{\geq}\inf_{\vec{0}\preceq\vec{u}\preceq \paren{T-1}\vec{1}} \frac{f_s\fp{\vec{u}+\bar{\vec{x}}_T} - f^*\fp{\nabla f_s\fp{\vec{u}+\vec{1}}} - f\fp{\vec{u}+\bar{\vec{x}}_T}}{f_s\fp{\vec{u}+\bar{\vec{x}}_T} - f\fp{\vec{u}+\bar{\vec{x}}_T}}
\\&\overset{\text{(i)}}{\geq}\inf_{\vec{0}\preceq\vec{u}\preceq \paren{T-1}\vec{1}} \frac{f_s\fp{\vec{u}} - f^*\fp{\nabla f_s\fp{\vec{u}+\vec{1}}} - f\fp{\vec{u}}}{f_s\fp{\vec{u}} - f\fp{\vec{u}}}
=:\beta_{f,f_s}.
\end{align*}
In equality (f), we replace $\bar{\boldsymbol{\lambda}}_T$ with $\nabla f_s\fp{\sum_{i=1}^{T-1}\bar{\vec{x}}_i+\vec{1}}$. Inequality (g) follows from replacing $\sum_{t=1}^Tv_t\fp{\bar{\vec{x}}_t}$ with $f_s\fp{\sum_{t=1}^T\bar{\vec{x}}_t}$ in the denominator. This creates a lower bound because Lemma \ref{lemma:psds_negative_seq_1} shows that the numerator is non-positive and Lemma \ref{lemma:objective_positive_seq_1} shows that $f_s\fp{\sum_{t=1}^T\bar{\vec{x}}_t} \leq \sum_{t=1}^Tv_t\fp{\bar{\vec{x}}_t}$.
Inequality (h) follows from observing that $\vec{0}\preceq \sum_{t=1}^{T-1}\bar{\vec{x}}_t\preceq \paren{T-1}\vec{1}$. Inequality (i) comes from seeing that $\frac{b-a}{b}$ is increasing in $b$ for $a,b>0$, where $f_s\fp{\vec{u}}-f\fp{\vec{u}} = b$ and $f^*\fp{\nabla f_s\fp{\vec{u}+\vec{1}}}=a$, and $f_s\fp{\vec{u}}-f\fp{\vec{u}}$ is increasing in $\vec{u}$ by assumption. Hence,
\[\beta_{f,f_s}P^{s1}\leq f_s\fp{\sum_{t=1}^T\bar{\vec{x}}_t} - f^*\fp{\bar{\boldsymbol{\lambda}}_T} - f\fp{\sum_{t=1}^T\bar{\vec{x}}_t}.\]
Define \[\alpha_{f,f_s} :=1-\beta_{f,f_s} = \sup_{\vec{0}\preceq\vec{u}\preceq \paren{T-1}\vec{1}} \frac{f^*\fp{\nabla f_s\fp{\vec{u}+\vec{1}}}}{f_s\fp{\vec{u}} - f\fp{\vec{u}}}.\]
Note that the assumption that $f_s\fp{\vec{u}}\geq f\fp{\vec{u}}$ for all $\vec{0}\preceq\vec{u}\preceq \paren{T-1}\vec{1}$ ensures that $\alpha_{f,f_s}\geq 0$, which ensures that the competitive ratio is non-negative. We lower bound $D^{s1}$ by $D^{\star}$ with Lemma \ref{lemma:dsdstar_seq_1} so $P^{s1} - D^{s1} \geq P^{s1}\beta_{f,f_s}$. Applying weak duality, we get that $P^{s1} - P^\star \geq P^{s1}\beta_{f,f_s}$. Rearranging this equation gives us the following:
\[\frac{P^{s1}}{P^\star} \geq \frac{1}{1-\beta_{f,f_s}} = \frac{1}{\alpha_{f,f_s}}.\]
This concludes the proof of Theorem \ref{thm:competitive_ratio_seq_1}.

\subsection{Proof of Theorem \ref{thm:quasiconvex_seq_1}}
\label{sec:proof_of_quasiconvex_seq_1}
In order to show that Problem \eqref{eq:quasiconvex_seq_1} is a quasiconvex optimization problem, we must verify that the constraints are convex and the objective is quasiconvex.
It suffices to show that
\[\max_{\vec{u}\in\U}~\frac{f^*\fp{\nabla f_s\fp{\vec{u}+\vec{1}}}}{f_s\fp{\vec{u}}-f\fp{\vec{u}}}\]
is a quasiconvex function in $\vec{a}$. Since a non-negative weighted maximum of quasiconvex functions is also quasiconvex, it suffices to show that $\frac{f^*\fp{\nabla f_s\fp{\vec{u}+\vec{1}}}}{f_s\fp{\vec{u}}-f\fp{\vec{u}}}$ is quasiconvex in $\vec{a}$ for a fixed $\vec{u}$. We can directly apply the definition of quasiconvexity. Let $S_{\alpha}\fp{f_s}$ be the sub-level sets of $f_s$ for $\vec{a}\in\R^N$. We have the following:
\begin{align*}
S_{\alpha}\fp{f_s}&=\setcond{\vec{a}\succeq\vec{1}}{\frac{f^*\fp{\nabla f_s\fp{\vec{u}+\vec{1}}}}{f_s\fp{\vec{u}}-f\fp{\vec{u}}}\leq\alpha}
\\&= \setcond{\vec{a}\succeq\vec{1}}{f^*\fp{\nabla f_s\fp{\vec{u}+\vec{1}}}\leq\alpha\paren{f_s\fp{\vec{u}}-f\fp{\vec{u}}}}.
\end{align*}
For a fixed value of $\vec{u}$, $\bracket{\nabla f_s\fp{\vec{u}+\vec{1}}}_d$ is linear in $\vec{a}$ for all $d$ and since $f^*$ is always convex, composing a convex function with a linear function of $\vec{a}$ is convex in $\vec{a}$. Finally, since $f_s\fp{\vec{u}}$ is linear in $\vec{a}$, the constraints of $S_{\alpha}\fp{f_s}$ are convex, and thus $S_{\alpha}\fp{f_s}$ is a convex set.

\end{document}